\documentclass[11pt,preprint]{imsart}
\usepackage{times}
\usepackage[margin=1in]{geometry}
\usepackage{amsfonts,amsthm,amsmath,amssymb,amstext,graphicx}
\usepackage{mathrsfs}
\RequirePackage[OT1]{fontenc}
\RequirePackage{natbib}
\RequirePackage[colorlinks,citecolor=blue,urlcolor=blue]{hyperref}

\allowdisplaybreaks 

\startlocaldefs

\numberwithin{equation}{section}

\newcommand{\Fc}{{ \mathscr{F}}}

\newcommand{\Bc}{{ \mathscr{B}}}

\newcommand{\msf}[1]{\mathsf{#1}}
\newcommand{\Pb}{{\msf{P}}}
\newcommand{\Eb}{{\msf{E}}}

 \newcommand{\PFA}{\mathsf{PFA}}
 \newcommand{\pfa}{\mathsf{PFA}}
\newcommand{\Hyp}{{\mathsf{H}}} 
\newcommand{\const}{\mathsf{const}}
\newcommand{\G}{{\msf{G}}}

\newcommand{\mc}[1]{\mathcal{#1}} 

\newcommand{\Nc}{{\mc{N}}}
\newcommand{\Ac}{{\mc{A}}}

\newcommand{\Lc}{{\mc{L}}}

\newcommand{\mb}[1]{\mathbf{#1}} 

\newcommand{\Xb}{{\mb{X}}}

\newcommand{\Yb}{\mb{Y}}

\newcommand{\Cb}{\mb{C}}

\newcommand{\mbb}[1]{\mathbb{#1}} 

\def\One{\mathchoice{\rm 1\mskip-4.2mu l}{\rm 1\mskip-4.2mu l}%
{\rm 1\mskip-4.6mu l}{\rm 1\mskip-5.2mu l}}
\newcommand\Ind[1]{{\One_{\{#1\}}}}

\newcommand{\class}{{\mbb{C}}_\alpha}
\newcommand{\Rbb}{{\mbb{R}}}
\newcommand{\Zbb}{{\mbb{Z}}}
\newcommand{\Nbb}{{\mbb{N}}}

\newcommand{\mrm}[1]{\mathrm{#1}}
\newcommand{\D}{{\mrm{d}}}

\newcommand{\wtT}{\widetilde{T}}
\newcommand{\wtX}{\widetilde{X}}
\newcommand{\wtS}{\widetilde{S}}

\newcommand{\abs}[1]{\left\vert#1\right\vert}
\newcommand{\set}[1]{\left\{#1\right\}}

\newcommand{\brc}[1]{\left(#1\right)}
\newcommand{\brcs}[1]{\left[#1\right]}

\newcommand{\xra}{\xrightarrow}
\newcommand{\esup}{\operatornamewithlimits{ess\,sup}}
\renewcommand{\le}{\leqslant} 
\renewcommand{\ge}{\geqslant}

\theoremstyle{plain}
\newtheorem{theorem}{Theorem}[section]
\newtheorem{lemma}{Lemma}[section]
\newtheorem{corollary}{Corollary}[section]

\newtheorem{definition}{Definition}[section]
\newtheorem{remark}{Remark}[section]

 \theoremstyle{remark}
\newtheorem{example}{Example}[section]

\endlocaldefs

\begin{document}

\begin{frontmatter}

\title{On Asymptotic Optimality in Sequential Changepoint Detection: Non-iid Case}

\runtitle{Sequential Changepoint Detection: Non-iid Case}

\begin{aug}

\author{\fnms{Alexander G.} \snm{Tartakovsky}\ead[label=e1]{alexg.tartakovsky@gmail.com}}
\address{
 Somers, CT, USA.\\
\printead{e1}}
\runauthor{Alexander Tartakovsky}

\affiliation{Somers, CT 06071}
\end{aug}

\begin{abstract}
We consider a sequential Bayesian changepoint detection problem for a general stochastic model, assuming that the observed data may be dependent and non-identically distributed and the prior 
distribution of
the change point is arbitrary, not necessarily geometric. Tartakovsky and Veeravalli (2004) developed a general asymptotic theory of changepoint detection in the non-iid case  and discrete time, and 
Baron and Tartakovsky (2006) in continuous time assuming the certain stability of the log-likelihood ratio process. This stability property was formulated in terms of the $r$-quick convergence 
of the normalized log-likelihood ratio process to a positive and finite number, which can be interpreted as the limiting Kullback--Leibler information between the ``change'' and ``no change'' hypotheses. 
In these papers, it was conjectured that the $r$-quick convergence can be relaxed in the $r$-complete  convergence, which is typically much easier to verify in particular examples. 
In the present paper, we justify 
this conjecture by showing that the Shiryaev change detection procedure is nearly optimal, minimizing asymptotically (as the probability of false alarm vanishes) the moments of the delay to 
detection up to order $r$
whenever $r$-complete convergence holds. We also study asymptotic properties of the Shiryaev--Roberts detection procedure in the Bayesian context.
\end{abstract}



 \received{\smonth{1} \syear{2016}}

\end{frontmatter}

\section{Introduction} \label{s:Intro}

In the beginning of the 1960s, \cite{Shiryaev61,ShiryaevTPA63} developed a Bayesian sequential changepoint detection (quickest disorder detection) theory in the 
iid case assuming that the observations are independent and identically distributed (iid) according to a distribution $F$ pre-change and another distribution $G$ post-change and with 
the prior distribution of the change point being geometric.
In particular, \cite{ShiryaevTPA63} proved that the detection procedure that is based on thresholding the posterior probability of the change being active before the current time is strictly optimal, 
minimizing the average delay to detection in the class of procedures with a given probability of false alarm. \cite{TV2005} generalized Shiryaev's theory 
for the non-iid case that covers very general discrete-time non-iid stochastic models and a wide class of prior distributions that include distributions with both exponential tails and heavy tails. 
In particular, it was proved that the Shiryaev detection procedure is asymptotically optimal -- it minimizes the average delay to detection 
as well as higher moments of the detection delay as the probability of 
a false alarm vanishes. \cite{BT2006} developed an asymptotic Bayesian theory for general continuos-time stochastic processes.

The key assumption in general asymptotic theories developed in \cite{BT2006,TV2005}  is a certain stability property of the log-likelihood ratio process between the 
``change'' and ``no-change'' hypotheses, which was expressed in the form of the strong law of large numbers with a positive and finite number and its strengthened $r$-quick  version. 
However, it is not easy (and in fact can be quite difficult) to verify $r$-quick convergence in particular applications and examples. 
For this reason, it was conjectured in \cite{BT2006,TV2005} that essentially the same asymptotic results may be obtained under a weaker
$r$-complete version of the strong law of large numbers for the log-likelihood ratio. In fact, in most examples provided in \cite{BT2006,TV2005} and in the recent book by 
\cite{TNBbook14}, verification of the $r$-quick convergence is replaced by verification of the $r$-complete convergence. 
The main goal of the present article is to confirm this conjecture, proving that the Shiryaev changepoint
detection procedure is asymptotically optimal under the $r$-complete convergence condition for the suitably normalized log-likelihood ratio process.

The rest of the article is organized as follows.  We formulate a general Bayesian changepoint detection problem and present some preliminary results in Section~\ref{sec:Problem}. 
In Section~\ref{sec:Gendiscr}, we consider the Shiryaev change detection procedure in detail and prove that it is asymptotically optimal under mild conditions associated with the $r$-complete convergence 
of the properly normalized log-likelihood ratio. In Section~\ref{sec:SR}, we discuss asymptotic properties and derive operating characteristics of another popular change detection procedure, 
the Shiryaev--Roberts procedure, and show that in general it is not asymptotically optimal in the Bayesian context, but preserves asymptotic optimality properties under certain conditions 
of the prior distribution. In Section~\ref{sec:Examples}, we provide examples of interesting cases where the conditions that we posit in Section~\ref{sec:Gendiscr} and Section~\ref{sec:SR} are satisfied. 
In Section~\ref{sec:Conclusion}, we conclude the paper by discussing the relevance of our results and providing additional remarks. 
Most of the proofs are presented in the main body of the paper, but proofs of some lemmas are given in the Appendix.

\section{Problem setup and preliminaries}\label{sec:Problem}

In the following, we deal only with discrete time $t=n\in \Zbb_+=\{0,1,2,\dots\}$. The continuous time case $t\in  \Rbb_+=[0,\infty)$ is 
more ``delicate'' and will be considered elsewhere. Having said that, let $(\Omega, \Fc ,\Fc_n, \Pb)$,  $n\in \Zbb_+$ be a filtered probability space,
where the sub-$\sigma$-algebra $\Fc_n =\sigma(\Xb^n)$ of $\Fc$ is assumed to be
generated by the process $\Xb^n = \{X_t\}_{1\le t \le n}$ observed up to time $n$. 
Let $\Pb_0$ and $\Pb_\infty$ be two probability measures defined on this space, which are assumed to be mutually locally absolutely continuos, 
so that the restrictions of these measures $\Pb_0^n$ and $\Pb_\infty^n$
to the sigma-algebras $\Fc_n$ are mutually absolutely continuous for all $n\ \ge 1$.

We are interested in the following changepoint problem. In a ``normal''
mode, the observed process $X_n$ follows the measure $\Pb_\infty$, and at an unknown time $\nu$ ($\nu\ge 0$) something happens and $X_n$ follows the measure
$\Pb_0$. The goal is to detect the change as soon as possible after it occurs, subject to a constraint on the risk of false alarms. The exact optimality criteria will be specified 
in Section~\ref{ssec:OptCriteria}.

\subsection{A general changepoint model}\label{ssec:CPDmodel}

Let $p_j(\Xb^n),~j=\infty,0$ denote densities of $\Pb_j^{n}$ (with respect 
to some non-degenerate $\sigma$-finite measure), where  $\Xb^n=(X_1,\dots,X_n)$ is the sample of size $n$. For a fixed $\nu\in\Zbb_+$, the change induces a probability measure 
$\Pb_\nu$ (correspondingly density $p_\nu(\Xb^n)=p(\Xb^n|\nu)$), which is a combination of the pre- and post-change densities:
\begin{equation} \label{noniidmodel}
p_\nu(\Xb^n)  = p_\infty(\Xb^{\nu}) \cdot p_0(\Xb_{\nu+1}^n | \Xb^{\nu}) 
= \prod_{i=1}^{\nu} p_{\infty}(X_i|\Xb^{i-1}) \cdot \prod_{i=\nu+1}^{n} p_{0}(X_i|\Xb^{i-1}) ,
\end{equation}
where $\Xb_m^n=(X_m,\dots,X_n)$ and $p_{j}(X_n|\Xb^{n-1})$ is the conditional density of $X_n$ given $\Xb^{n-1}$. In the sequel we assume that $\nu$ is the serial number of the last pre-change 
observation.  Note that in general the conditional densities $p_0(X_i | \Xb^{i-1})$, $i =\nu+1, \nu+2, \dots$ may depend 
on the changepoint $\nu$, i.e.,  $p_0(X_i | \Xb^{i-1}) =p_0^{(\nu)}(X_i | \Xb^{i-1})$ for $i > \nu$. Certainly the 
densities $p_{j}(X_i|\Xb^{i-1})=p_{j,i}(X_i|\Xb^{i-1})$, $j=0,\infty$ may depend on $i$.

In a particular {\em iid case}, addressed in detail in the past  the observations are independent and identically distributed (iid) with density $f_\infty(x)$ in the normal (pre-change) 
mode and with another density $f_0(x)$ in the abnormal 
(post-change) mode, i.e., in this case, \eqref{noniidmodel} holds with $p_{\infty}(X_i|\Xb^{i-1})=f_\infty(X_i)$ and $p_0(X_i|\Xb^{i-1})=f_0(X_i)$.

We are interested in a Bayesian setting where the change point $\nu$ is assumed
to be a random variable independent of the observations with prior probability distribution $\Pi_n=\Pb(\nu\le n)$, $n\in\Zbb_+$. We also write $\pi_k=\Pb(\nu=k)$ for the probability on non-negative integers,
$k=0,1,2,\dots$. Formally, we allow the change point $\nu$ 
to take negative values too, but the detailed distribution for $k<0$ is not important. The only value we need is the cumulative probability $q= \Pb(\nu<0)$. The probability $\Pb(\nu \le 0)=q+\pi_0$ 
is the probability of the ``atom'' associated with the event that the change already took place before the observations became available. 

In the past, the typical choice for the prior  distribution was (zero modified) geometric distribution,
\begin{equation}\label{Geom}
\Pb(\nu<0)=q \quad \text{and} \quad \Pb(\nu=k) =(1-q) \rho (1-\rho)^k \quad \text{for}~ k=0,1,2,\dots,
\end{equation}
where $0\le q <1$, $0<\rho <1$. 

In the rest of the paper, we consider an arbitrary prior distribution that belongs to the class of distributions that satisfy the following condition:
\vspace{2mm}

\noindent $\Cb$. {\em  For some $0 \le \mu <\infty$,
\begin{equation}\label{Prior}
\lim_{n\to\infty}\frac{|\log(1-\Pi_n)|}{n} = \mu.
\end{equation}
In the case that $\mu=0$, we assume in addition that for some $r\ge 1$
\begin{equation}\label{Prior1}
\sum_{k=0}^\infty \pi_k |\log\pi_k|^r < \infty .
\end{equation}}
If $\mu >0$, then the prior distribution has an exponential right tail. Such distributions, as geometric and discrete versions of  gamma and logistic distributions, i.e., models with bounded hazard
rate, belong to this class. In this case, condition \eqref{Prior1} holds automatically. If $\mu=0$, then the distribution has a heavy tail, i.e., such a distribution belongs to the model with a vanishing hazard rate.
However, we cannot allow this distribution to have a tail that is too heavy, which is guaranteed by condition \eqref{Prior1}.

\subsection{Optimality criteria}\label{ssec:OptCriteria}

Any sequential detection procedure is a stopping time $T$ for the observed process $\{X_n\}_{n\in\Zbb_+}$, i.e., $ T$ is
an extended random variable, such that the event $\{ T = n\}$ belongs to the sigma-algebra $\Fc_n$.  A false alarm is raised whenever
$ T\le \nu$. A good detection procedure should guarantee a small delay to
detection $ T-\nu$ provided that there is no false alarm, while the rate (or risk) of false alarms should be kept at a given, usually low level.

Let $\Pb_k$ and $\Eb_k$ denote the probability and the corresponding expectation when the change occurs at time $\nu=k \in \Zbb_+$. 
In what follows, $\Pb^\pi$ denotes the probability measure on the Borel sigma-algebra in $\Rbb^{\infty}\times \Nbb$ defined as 
$\Pb^\pi(\Ac\times J)=\sum_{k\in J}\,\pi_{k} \Pb_{k}\left(\Ac\right)$ for $\Ac\in \Bc(\Rbb^{\infty})$, $J\subseteq\Nbb$ and $\Eb^\pi$ denotes the expectation with respect to $\Pb^\pi$.

In a Bayesian setting, the risk associated with the delay to detection  is usually measured by the average
delay to detection
\begin{equation} \label{ADDdefdiscr}
\Eb^\pi( T-\nu| T> \nu)= 
\frac{\sum_{k=0}^\infty \pi_k \Eb_k( T-k| T >k) \Pb_\infty( T >k)}{1-\PFA( T)} 
\end{equation}
and the risk associated with a false alarm by the weighted probability of false
alarm (PFA) defined as 
\begin{equation} \label{PFAdefdiscr}
\pfa( T)=\Pb^\pi( T \le \nu)= \sum_{k=1}^\infty \pi_k \Pb_\infty( T \le k) .
\end{equation}
In \eqref{ADDdefdiscr} and \eqref{PFAdefdiscr} we use the fact that $\Pb_k( T > k) = \Pb_\infty( T > k)$ and $\Pb_k( T \le k) = \Pb_\infty( T \le k)$ for $k\in \Zbb_+$ and that $\Pb_\infty( T\le 0) =0$. 

For $0<\alpha <1$, let  $\class = \set{ T: \PFA( T) \le \alpha}$ be a class of detection procedures for which the weighted probability of false alarm does not exceed the predefined level $\alpha$. 
In a Bayesian setting, the goal is to find an optimal procedure that minimizes in the class $\class$ the average delay to detection, i.e.,
\[
\text{find}~~  T_{\rm opt}\in \class ~~ \text{such that} ~~ \Eb^\pi( T_{\rm opt}-\nu| T_{\rm opt} > \nu) = \inf_{ T\in\class}\Eb^\pi( T-\nu| T> \nu) .
\]
However, except for the iid case, the solution of this problem is not tractable. For this reason, we address the asymptotic problem of minimizing the average detection delay 
as $\alpha$ approaches zero. For practical
purposes, it is also interesting to consider the problem of minimizing higher moments of the detection delay $\Eb^\pi[( T-\nu)^m|  T >\nu]$ for some $m\ge 1$, i.e., to find a first-order asymptotically optimal
detection procedure $ T_{\rm o}\in\class$ that satisfies
\begin{equation}\label{FOAOcriterion}
\lim_{\alpha\to0} \frac{\Eb^\pi[( T_{\rm o}-\nu)^m|  T_{\rm o} >\nu]}{\inf_{ T\in\class}\Eb^\pi[( T-\nu)^m|  T >\nu]} =1.
\end{equation}

\subsection{Change detection procedures}\label{ssec:Procedures}

Let ``$\Hyp_k: \nu=k$'' and ``$\Hyp_\infty: \nu=\infty$'' be the hypotheses that the change occurs at the point $0\le k<\infty$ and that the change never happens, respectively. 
Then, using \eqref{noniidmodel}, we obtain that the likelihood ratio (LR) between these hypotheses when the sample $\Xb^n=(X_1,\dots,X_n)$ is observed is
\[
 \frac{\D \Pb_k^n}{\D \Pb_\infty^n} = \prod_{i=k+1}^n \frac{p_0(X_i|\Xb^{i-1})}{p_\infty(X_i|\Xb^{i-1})}, \quad k <n .
\]
Write $\Lc_i =p_0(X_i|\Xb^{i-1})/p_\infty(X_i|\Xb^{i-1})$ and introduce the normalized average (weighted) LR
\[
\Lambda_n = \frac{1}{\Pb(\nu \ge n)} \brc{q \prod_{i=1}^n \Lc_i + \sum_{k=0}^{n-1} \pi_k \prod_{i=k+1}^{n} \Lc_i}, \quad n \in \Zbb_+.
\]
Note that $\Lambda_0=q/(1-q)$. Let $g_n=\Pb(\nu < n | \Xb^n)$ stand for the posterior probability of the change being in effect up to time $n$.  \cite{ShiryaevTPA63} proved that, 
in the iid case, the detection procedure $T_a= \inf\set{n: g_n \ge a}$ is strictly optimal for every $0<\alpha<1$ -- it minimizes the average detection delay $\Eb^\pi( T-\nu| T> \nu)$ if $a=a_\alpha$ 
is selected so that $\PFA(T_a)=\alpha$ and the prior distribution is geometric. 
As in \cite{TNBbook14, TV2005}, we refer to this procedure as the Shiryaev detection procedure in the general non-iid case too. We now show that $\Lambda_n = g_n/(1-g_n)$, 
so that the Shiryaev procedure can be written as
\begin{equation}\label{ShirProc}
T_A = \inf\set{n \ge 1: \Lambda_n \ge A}, \quad A >0.
\end{equation}
Indeed, $g_n =\sum_{k=-\infty}^{n-1} \Pb(\nu=k|\Xb^n)$, where 
\[
\begin{aligned}
\Pb(\nu=k|\Xb^n) & = \frac{\pi_k \prod_{j=1}^k p_\infty(X_i|\Xb^{i-1}) \prod_{i=k+1}^{n} p_0(X_i|\Xb^{i-1})}{\sum_{k=-\infty}^\infty \pi_k \prod_{j=1}^k p_\infty(X_i|\Xb^{i-1}) \prod_{i=k+1}^{n} p_0(X_i|\Xb^{i-1})}
\\
& = \frac{\pi_k \prod_{i=k+1}^{n} \Lc_i}{q \prod_{i=1}^n \Lc_i + \sum_{k=0}^{n-1} \pi_k \prod_{i=k+1}^{n} \Lc_i + \Pb(\nu\ge n)},
\end{aligned}
\]
and we obtain
\[
g_n= \frac{q \prod_{i=1}^n \Lc_i + \sum_{k=0}^{n-1}\pi_k \prod_{i=k+1}^{n} \Lc_i}{q \prod_{i=1}^n \Lc_i + \sum_{k=0}^{n-1} \pi_k \prod_{i=k+1}^{n} \Lc_i + \Pb(\nu\ge n)} .
\]
Therefore, 
\[
\frac{g_n}{1-g_n} = \frac{1}{\Pb(\nu \ge n)} \brc{q \prod_{i=1}^n \Lc_i + \sum_{k=1}^{n-1} \pi_k \prod_{i=k+1}^{n} \Lc_i} =\Lambda_n.
\]
In particular, in the popular case of zero modified geometric prior \eqref{Geom}, the statistic $\Lambda_n$ is
\begin{equation}\label{LambdaGeom}
\Lambda_n = \frac{q}{1-q} \prod_{i=1}^n \brc{\frac{\Lc_i}{1-\rho}}  + \rho \sum_{k=1}^n \prod_{i=k}^n \brc{\frac{\Lc_i}{1-\rho}} .
\end{equation}

In the following, to avoid triviality, we assume that $A>q/(1-q)$, since otherwise $T_A=0$ with probability $1$.

By Lemma~7.2.1 in \cite{TNBbook14},
\begin{equation}\label{PFAineq}
\PFA(T_A) \le 1/(1+A) \quad \text{for every}~ A>  q/(1-q),
\end{equation}
 and therefore setting $A=A_\alpha=(1-\alpha)/\alpha$ guarantees that $T_A\in \class$.
 
 Another popular change detection procedure is the Shiryaev--Roberts (SR) procedure (due to \cite{ShiryaevTPA63} and \cite{Roberts66}) given by the stopping time
 \begin{equation}\label{SRproc}
 \wtT_B=\inf\set{n\ge 1: R_n \ge B}, \quad B>0,
 \end{equation}
 where the statistic $R_n$, the SR statistic, is given by
 \begin{equation}\label{SRstat}
 R_n = \sum_{k=1}^n \prod_{i=k}^n \Lc_i, \quad n \ge 0 ~~ (R_0=0). 
 \end{equation}
The statistic $R_n$ can be viewed as a limit of the statistic $\Lambda_n/\rho$ as $\rho\to 0$  when the prior distribution of the change point is geometric \eqref{Geom} with $q=0$. 
Indeed, see \eqref{LambdaGeom}. 

\subsection{$r$-Quick convergence versus $r$-complete convergence} \label{ssec:rcomplete}

Introduce the LLRs
\[
Z_i = \log \frac{p_0(X_i|\Xb^{i-1})}{p_\infty(X_i|\Xb^{i-1})}, \quad \lambda_{k+n}^k = \log \frac{\D\Pb_k^{k+n}}{\D\Pb_\infty^{k+n}} = \sum_{i=k+1}^{k+n} Z_i , \quad n \ge 1.
\]

We need the following two definitions. 

\begin{definition}[{\rm \cite{Lai1976,TNBbook14}}]\label{def:rquick}
Let $r>0$. For $k=0,1,2,\dots$, we say that the normalized LLR $n^{-1}  \lambda_{k+n}^k$ converges $r$-quickly to a constant $I$ as $n\to\infty$ under probability $\Pb_k$ if $\Eb_k [L_k(\varepsilon)]^r <\infty$ for all 
$\varepsilon>0$, where $L_k(\varepsilon) = \sup\set{n\ge 1: |n^{-1}  \lambda_{k+n}^k-I |> \varepsilon}$ ($\sup\{\varnothing\}=0$) is the last time when $n^{-1}  \lambda_{k+n}^k$ 
leaves the interval $[I-\varepsilon, I+\varepsilon]$. 
\end{definition}

\begin{definition}[{\rm \cite{TNBbook14}}]\label{def:rcomplete}
Let $r>0$. For $k=0,1,2,\dots$, we say that the normalized LLR $n^{-1}  \lambda_{k+n}^k$ converges $r$-completely to a constant $I$ as $n\to \infty$ under probability $\Pb_k$ 
if  for all $\varepsilon>0$, 
\begin{equation}\label{rcompletedef}
\sum_{n=1}^\infty n^{r-1} \Pb_k\set{\abs{n^{-1}\lambda_{k+n}^k - I} > \varepsilon} < \infty.
\end{equation}
(For $r=1$ this mode of convergence was introduced by \cite{HsuRobbins47}.)
\end{definition}

Note first that in general $r$-quick convergence is a stronger property than $r$-complete convergence. See Lemma~2.4.1 in \cite{TNBbook14}. More importantly, checking $r$-quick convergence in applications
is often much more difficult than checking $r$-complete convergence.

In the discrete time case, \cite{TV2005} developed a general asymptotic Bayesian theory of changepoint detection assuming that the LLR obeys the strong law of large numbers (SLLN) with some positive and 
finite constant $I$, i.e.,
\begin{equation}\label{SLLN}
\frac{1}{n} \lambda_{k+n}^k \xra[n\to\infty]{\Pb_k-{\rm a.s.}} I \quad \text{for all}~ k\in \Zbb_+,
\end{equation}
with a certain rate of convergence expressed via the $r$-quick convergence, specifically assuming in addition that for some $r\ge 1$
\begin{equation}\label{rquickaver}
\sum_{k=0}^\infty \pi_k \Eb_k [L_k(\varepsilon)]^r <\infty.
\end{equation}
A similar development was performed by \cite{BT2006} in continuos time, assuming that
\[
\int_0^\infty\Eb_u [L_u(\varepsilon)]^r \, \D \Pi_u <\infty.
\]
However, as we already mentioned, verification of the latter $r$-quick convergence condition in particular examples is not an easy task. 

In \cite{BT2006,TV2005}, it was conjectured that all asymptotic results, including near optimality of the Shiryaev procedure (in the sense defined in \eqref{FOAOcriterion}), 
hold if the $r$-quick convergence condition \eqref{rquickaver} is weakened into
the $r$-complete convergence
\[
\sum_{k=0}^\infty \pi_k \brcs{\sum_{n=1}^\infty n^{r-1} \Pb_k\set{\abs{n^{-1}\lambda_{k+n}^k - I} > \varepsilon}} < \infty
\]
(with an obvious modification in continuous time). In the following sections, we justify this conjecture.

\section{Asymptotic operating characteristics and optimality of the Shiryaev procedure}\label{sec:Gendiscr} 

In this section, we present the main results related to asymptotic optimality of the Shiryaev detection procedure in the general non-iid case as well as in the case of independent observations.

\subsection{The non-iid case}\label{ssec:Gencase}

The following lemma, that establishes the asymptotic lower bounds for moments of the detection delay, will be used throughout  the paper. While its proof may be found in \cite{TV2005}, 
parts of the proof are scattered in \cite{TV2005}, and for the sake of convenience we provide a sketch of the improved version of the proof  in the Appendix.  

\begin{lemma}\label{Lem:LB}
Let $T_A$ be the Shiryaev changepoint detection procedure defined in \eqref{ShirProc}. Let, for some $\mu\ge 0$, the prior distribution of the change point satisfy condition \eqref{Prior}. 
Assume that for some positive and finite $I$ 
\begin{equation}\label{Pmaxgen}
\lim_{M\to\infty} \Pb_k\brc{\frac{1}{M} \max_{1 \le n \le M}\lambda_{k+n}^k \ge (1+\varepsilon) I} = 0 \quad \text{for all}~ \varepsilon >0~ \text{and all}~ k \in \Zbb_+ .
\end{equation}  
Then, for all $m>0$, 
\begin{equation}\label{LBinclass}
\liminf_{\alpha\to0} \frac{\inf_{ T\in\class} \Eb^\pi\brcs{\brc{ T-\nu}^m| T >\nu}}{|\log \alpha|^m} \ge  \frac{1}{(I+\mu)^m}.
\end{equation}
and
\begin{equation}\label{LBTA1}
\liminf_{A\to\infty} \frac{\Eb^\pi\brcs{\brc{T_A-\nu}^m|T_A >\nu}}{(\log A)^m} \ge  \frac{1}{(I+\mu)^m}.
\end{equation}
\end{lemma}

Define
\begin{equation}\label{Upsilon}
\Upsilon_{k,r}(\varepsilon) = \sum_{n=1}^\infty n^{r-1} \Pb_k\brc{\frac{1}{n} \lambda_{k+n}^k < I  - \varepsilon} .
\end{equation}
Recall that by \eqref{PFAineq}, $\PFA(T_A) \le (1+A)^{-1}$ for any $0<A<q/(1-q)$, which implies that $\PFA(T_{A_\alpha}) \le \alpha$ (i.e., $T_{A_\alpha}\in \class$) for any $0<\alpha< 1-q$ 
if $A=A_\alpha=(1-\alpha)/\alpha$.

The following theorem is the main result in the general non-iid case, which shows that the Shiryaev detection procedure is asymptotically optimal under mild conditions for the observations and prior distributions.
 
\begin{theorem}\label{Th:FOAOgen}
Let $T_A$ be the Shiryaev changepoint detection procedure defined in \eqref{ShirProc}. Let $r\ge 1$ and let the prior distribution of the change point satisfy condition {\rm ($ \Cb$)}. Assume that for some 
number $0<I<\infty$ condition \eqref{Pmaxgen} is satisfied and  that the following condition holds as well \begin{equation}\label{rcompleteleftgen}
\sum_{k=0}^\infty \pi_k \Upsilon_{k,r}(\varepsilon)  <\infty \quad \text{for all}~ \varepsilon >0 .
\end{equation}  

\noindent {\rm \bf (i)} Then for all $0<m \le r$
\begin{equation} \label{MADDAOgen1}
\lim_{A\to\infty}  \frac{\Eb^\pi[(T_A-\nu)^m | T_A > \nu]}{(\log A)^m} =\frac{1}{(I+\mu)^m}.
\end{equation}

\noindent {\rm \bf (ii)}   If $A=A_\alpha=(1-\alpha)/\alpha$, where $0<\alpha<1-q$, then $T_{A_\alpha}\in \class$ and it is asymptotically optimal as $\alpha\to0$ in class $\class$, minimizing moments of the detection delay up to order $r$, i.e., for all $0<m \le r$,
\begin{equation}\label{FOAOmomentsgen}
\lim_{\alpha\to0} \frac{\inf_{T \in \class} \Eb^\pi[(T-\nu)^m | T > \nu] }{\Eb^\pi[(T_{A_\alpha}-\nu)^m | T_{A_\alpha} > \nu]} = 1 .
\end{equation}
Also, the following first-order asymptotic approximations hold:
\begin{equation} \label{MADDAOgen}
\inf_{T \in \class} \Eb^\pi[(T-\nu)^m | T> \nu] \sim  \Eb^\pi[(T_{A_\alpha}-\nu)^m | T_{A_\alpha} > \nu] \sim \brc{\frac{|\log\alpha|}{I+\mu}}^m \quad \text{as $\alpha \to 0$}.
\end{equation}
This assertion also holds if $A=A_\alpha$ is selected so that $\PFA(T_{A_\alpha}) \le \alpha$ and $\log A_\alpha\sim |\log\alpha|$ as $\alpha\to0$.
\end{theorem}

In order to prove this theorem we need the following lemma, the proof of which is given in the Appendix.

\begin{lemma}\label{Lem:UpperEk}
Let $r\ge 1$ and let the prior distribution of the change point satisfy condition {\rm ($ \Cb$)}. Then for a sufficiently large $A$, any $0<\varepsilon <I+\mu$ and all $k\in\Zbb_+$,
\begin{equation}\label{Ekineq}
\Eb_k[\brc{T_A-k}^+]^r \le  \brc{1+\frac{\log (A/\pi_k)}{I+\mu-\varepsilon}}^r + r 2^{r-1} \, \sum_{n=1}^\infty n^{r-1} \Pb_k\brc{\frac{1}{n} \lambda_{k+n}^k < I  - \varepsilon-\delta},
\end{equation}
where $\delta\to0$ as $A\to\infty$. 

If the prior distribution is geometric \eqref{Geom} with $q=0$, i.e., $\pi_k=\rho(1-\rho)^k$, $k\in\Zbb_+$, then for any $A>0$, any $0<\varepsilon <I+\mu$ and all $k\in\Zbb_+$
\begin{equation}\label{EkineqGeom}
\Eb_k[\brc{T_A-k}^+]^r \le  \brc{1+\frac{\log (A/\rho)}{I+\mu-\varepsilon}}^r + r 2^{r-1} \, \sum_{n=1}^\infty n^{r-1} \Pb_k\brc{\frac{1}{n} \lambda_{k+n}^k < I  - \varepsilon},
\end{equation}
where $\mu=-\log(1-\rho)$.
\end{lemma}

\begin{proof}[Proof of Theorem~\ref{Th:FOAOgen}]
(i)  By Lemma~\ref{Lem:LB}, under the right-tail condition \eqref{Pmaxgen} the asymptotic lower bound \eqref{LBTA1} holds for all $m>0$.
Thus, to establish \eqref{MADDAOgen1} it suffices to show that, under the left-tail condition \eqref{rcompleteleftgen},
\begin{equation}\label{UppergenA}
\limsup_{A\to\infty} \frac{\Eb^\pi[(T_{A}-\nu)^r | T_{A} > \nu]}{(\log A)^r} \le \frac{1}{(I+\mu)^r} .
\end{equation}

 Let $\varepsilon_1=\varepsilon + \delta$. By Lemma~\ref{Lem:UpperEk}, for any $0<\varepsilon < I+\mu$, 
\begin{equation}\label{UpperADDA}
\begin{aligned}
\Eb^\pi[(T_{A}-\nu)^r | T_{A} > \nu] & = \frac{\sum_{k=0}^\infty \pi_k  \Eb_{k}\brcs{(T_A-k)^+}^r}{1-\PFA(T_A)} 
\\
&\le\frac{ \sum_{k=0}^\infty  \pi_k \brc{1+\frac{\log (A/\pi_k)}{I+\mu-\varepsilon}}^r + r 2^{r-1}\, \sum_{k=0}^\infty \pi_k  \Upsilon_{k,r}(\varepsilon_1)}{A/(1+A)} ,
\end{aligned}
\end{equation}
where we used the inequality $1-\PFA(T_A) \ge A/(1+A)$. By conditions \eqref{rcompleteleftgen} and ($ \Cb$),
\[
 \sum_{k=0}^\infty \pi_k  \Upsilon_{k,r}(\varepsilon_1) < \infty \quad \text{for any}~ \varepsilon_1 >0 \quad \text{and} \quad  \sum_{k=0}^\infty \pi_k |\log\pi_k|^r < \infty,
\]
which implies that
\[
\Eb^\pi[(T_{A}-\nu)^r | T_{A} > \nu] \le \brc{\frac{\log A}{I+\mu-\varepsilon}}^r (1+o(1)) \quad \text{as}~ A \to \infty.
\]
Since $\varepsilon$ can be arbitrarily small, the upper bound \eqref{UppergenA} follows and the proof of (i) is complete.

(ii) Setting $A=A_\alpha=(1-\alpha)/\alpha$ in \eqref{MADDAOgen1} yields 
\begin{equation}\label{MOapproxgenalpha}
\lim_{\alpha\to0} \frac{\Eb^\pi[(T_{A_\alpha}-\nu)^r | T_{A_\alpha} > \nu]}{|\log\alpha|^r} = \frac{1}{(I+\mu)^r},
\end{equation}
which along with the lower bound \eqref{LBinclass}  in Lemma~\ref{Lem:LB} completes the proof of \eqref{FOAOmomentsgen}. Finally, asymptotic approximations \eqref{MADDAOgen} 
follow from \eqref{MOapproxgenalpha} and \eqref{FOAOmomentsgen}.  Evidently, 
\eqref{MOapproxgenalpha} and \eqref{FOAOmomentsgen}, and therefore, approximations \eqref{MADDAOgen} also hold if threshold $A_\alpha$ is chosen so that $T_{A_\alpha}\in\class$ and
$\log A_\alpha\sim |\log \alpha|$ as $\alpha\to0$. The proof is complete.
\end{proof}

Theorem~\ref{Th:FOAOgen} implies that the Shiryaev procedure $T_A$ is asymptotically optimal whenever the LLR converges to a constant $I$ $r$-completely. Indeed, we have the following corollary.

\begin{corollary}\label{Cor: Cor 1} 
Let $r \ge 1$.  Let the prior distribution of the change point satisfy condition {\rm ($ \Cb$)}. Assume that for some $0<I<\infty$
\begin{equation}\label{rcompletegen}
\sum_{k=0}^\infty \pi_k \brcs{\sum_{n=1}^\infty n^{r-1} \Pb_k\brc{\abs{\frac{1}{n}\lambda_{k+n}^k - I} > \varepsilon}} <\infty \quad \text{for all}~ \varepsilon >0 .
\end{equation}  
Then \eqref{MADDAOgen1}, \eqref{FOAOmomentsgen} and \eqref{MADDAOgen} hold true.
\end{corollary}

 \begin{proof}
Obviously, the $r$-complete convergence condition \eqref{rcompletegen} implies both conditions \eqref{Pmaxgen} and \eqref{rcompleteleftgen}, which immediately proves  the assertion of the corollary. 
 \end{proof}
 
Theorem~\ref{Th:FOAOgen} is very general and covers, perhaps, almost all possible non-iid models as well as a large class of prior distributions. However, note that condition ($ \Cb$) does 
not include the case where $\mu$ is strictly positive, $\mu>0$, but may go to zero, $\mu\to 0$. Indeed, in this case, the sum in \eqref{Prior1} approaches infinity, and the results of 
the theorem are not applicable in general. To see this, it suffices to consider the geometric prior \eqref{Geom} with $q=0$. Then $\mu=|\log(1-\rho)|$ and
 \[
\varkappa(\rho) := \sum_{k=0}^\infty \pi_k |\log \pi_k| = |\log\rho| +(1-\rho)|\log(1-\rho)|/\rho \sim |\log \rho| \quad \text{as}~ \rho \to 0.
 \]
 For $r=1$, inequality \eqref{UpperADDA} has the form
 \[
\Eb^\pi(T_{A}-\nu | T_{A} > \nu) \le\frac{1+\frac{\log A + \varkappa(\rho)}{I+\mu-\varepsilon} +  \sum_{k=0}^\infty \rho(1-\rho)^k  \Upsilon_{k,1}(\varepsilon_1)}{A/(1+A)} .
 \]
Clearly, the upper bound \eqref{UppergenA} holds if, and only if, $\rho=\rho_A$ decays in such a way that $|\log\rho_A|=o(\log A)$. Otherwise the argument breaks down and the results are not correct.

In the next lemma, which is analogous to Lemma~\ref{Lem:LB}, we provide an asymptotic lower bound for moments of the detection delay in class $\class$ when the prior distribution  
$\pi^\alpha=\{\pi_k^\alpha\}$ of the change point may depend on the PFA constraint $\alpha$ and becomes ``flat'' when $\alpha$ vanishes. Its proof is given in the Appendix.

\begin{lemma}\label{Lem:LB2}
Let $T_A$ be the Shiryaev changepoint detection procedure defined in \eqref{ShirProc}. Let the prior distribution $\pi^\alpha=\{\pi_k^\alpha\}$ of the change point satisfy condition \eqref{Prior} with 
$\mu> 0$ such that $\mu=\mu_\alpha\to 0$ as $\alpha\to 0$.  Assume that for some $0<I<\infty$ condition \eqref{Pmaxgen} holds. 
Then, for all $m>0$, 
\begin{equation}\label{LBinclass2}
\liminf_{\alpha\to0} \frac{\inf_{ T\in\class} \Eb^{\pi^\alpha}\brcs{\brc{ T-\nu}^m| T >\nu}}{|\log \alpha|^m} \ge  \frac{1}{I^m}
\end{equation}
and
\begin{equation}\label{LBTA2}
\liminf_{A\to\infty} \frac{\Eb^{\pi^\alpha}\brcs{\brc{T_A-\nu}^m|T_A >\nu}}{(\log A)^m} \ge  \frac{1}{I^m}.
\end{equation}
\end{lemma}

Using this lemma, we now establish asymptotic optimality of the Shiryaev procedure in the case where $\mu=\mu_\alpha$ approaches zero as $\alpha\to0$.

\begin{theorem}\label{Th:FOAOgen2}
Let $r\ge 1$.  Suppose that the prior distribution $\pi^\alpha=\{\pi_k^\alpha\}$ of the change point $\nu$ satisfies condition \eqref{Prior} with $\mu=\mu_\alpha\to 0$ as $\alpha\to 0$ and that $\mu_\alpha$
approaches zero at such rate that 
\begin{equation}\label{Prior3}
\lim_{\alpha\to 0} \frac{{\sum_{k=0}^\infty |\log \pi_k^\alpha|^r \pi_k^\alpha}}{|\log \alpha|^r} = 0.
\end{equation} 
Assume that for some $0<I<\infty$ condition \eqref{Pmaxgen} and the following uniform $r$-complete convergence 
\begin{equation}\label{uniformrcomplete}
\sup_{0\le k<\infty} \Upsilon_{k,r}(\varepsilon) <\infty
\end{equation}
are satisfied.  If $A=A_\alpha$ is so selected that $\PFA(T_{A_\alpha}) \le \alpha$ and 
$\log A_\alpha\sim |\log\alpha|$ as $\alpha\to0$, 
in particular $A_\alpha=(1-\alpha)/\alpha$, then the Shiryaev procedure $T_{A_\alpha}$, given by \eqref{ShirProc}, is asymptotically optimal as $\alpha\to0$ in class $\class$, 
minimizing moments of the detection delay up to order $r$: for all $0<m \le r$
\begin{equation} \label{MADDAOgen2}
\inf_{T \in \class} \Eb^{\pi^\alpha}[(T-\nu)^m | T> \nu] \sim  \Eb^{\pi^\alpha}[(T_{A_\alpha}-\nu)^m | T_{A _\alpha}> \nu] \sim \brc{\frac{|\log\alpha|}{I}}^m \quad \text{as}~ \alpha \to 0.
\end{equation}
\end{theorem}

\begin{proof}
Substituting $A=(1-\alpha)/\alpha$ (or $\log A_\alpha \sim |\log \alpha|$) in inequality \eqref{UpperADDA}, we obtain
\[
\Eb^{\pi^\alpha}[(T_{A}-\nu)^r | T_{A}  > \nu] 
\le\frac{ \sum_{k=0}^\infty  \pi_k^\alpha \brc{1+\frac{\log ((1-\alpha)/\alpha\pi_k^\alpha)}{I+\mu_\alpha-\varepsilon}}^r + r 2^{r-1}\, \sup_{k\ge 0}\Upsilon_{k,r}(\varepsilon_1)}{1-\alpha} .
\]
Using conditions \eqref{Prior3} and \eqref{uniformrcomplete} and taking into account that $\mu_\alpha\to0$ as $\alpha\to0$ yields
\[
\Eb^{\pi^\alpha}[(T_{A_\alpha}-\nu)^m | T_{A _\alpha}> \nu]  \le  \brc{\frac{|\log\alpha|}{I-\varepsilon}}^m (1+o(1)) \quad \text{as}~ \alpha \to 0.
\]
Since $\varepsilon$ is an arbitrary number in $(0,I)$,  we obtain the asymptotic upper bound 
\[
\Eb^{\pi^\alpha}[(T_{A_\alpha}-\nu)^m | T_{A _\alpha}> \nu]  \le  \brc{\frac{|\log\alpha|}{I}}^m (1+o(1)) \quad \text{as}~ \alpha \to 0,
\]
which along with the lower bound \eqref{LBinclass2} in Lemma~\ref{Lem:LB2}  proves \eqref{MADDAOgen2}. 
\end{proof}

 \begin{remark}
If the prior distribution is geometric \eqref{Geom}, then Theorem~\ref{Th:FOAOgen2} holds whenever the parameter $\rho=\rho_\alpha\to0$ at a rate $|\log\rho_\alpha|=o(|\log\alpha|)$. Indeed, see the upper bound 
\eqref{EkineqGeom} in Lemma~\ref{Lem:UpperEk}.
\end{remark}

 \subsection{The case of independent observations}\label{ssec:Indepcase}

The results of the previous subsection show that the lower bounds \eqref{LBinclass} and \eqref{LBTA1} for moments of the detection delay hold whenever the LLR process $\lambda_{k+n}^k$ 
obeys the SLLN \eqref{SLLN}, since in this case condition \eqref{Pmaxgen} is satisfied. However, in general, an almost sure convergence  \eqref{SLLN} is not sufficient for obtaining the upper bounds, and therefore, 
for asymptotic optimality of the Shiryaev procedure. In fact, this condition does not even guarantee finiteness of the average delay to detection $\Eb^\pi(T_A-\nu|T_A>\nu)$, and to obtain
meaningful results we need to strengthen the SLLN into the $r$-complete version.  On the other hand, in the iid case,
where conditioned on $\nu=k$ the observations $X_1,\dots,X_k$ are iid with pre-change density $f_\infty(x)$ and $X_{k+1}, X_{k+2},\dots$ are iid with post-change density $f_0(x)$, 
the situation is dramatically different. By Theorem~4 of \cite{TV2005}, the Shiryaev procedure asymptotically (as $\alpha\to0$) minimizes all positive moments of the detection delay in 
class $\class$ if the prior distribution is geometric and the Kullback--Leibler information number
\begin{equation}\label{KL}
\mc{K} = \Eb_0 \lambda_1^0 = \int \log \brc{\frac{f_0(x)}{f_\infty(x)}}\, \D \mu(x)
\end{equation}
is positive and finite.
 
We now extend this result to the case where observations are independent, but not necessarily identically distributed, i.e., 
$p_\infty(X_i|\Xb^{i-1})=f_{\infty, i}(X_i)$ and $p_0(X_i|\Xb^{i-1})=f_{0, i}(X_i)$ in \eqref{noniidmodel}. More generally, we may assume that the increments $Z_i$ of the LLR 
$\lambda_n^k=\sum_{i=k+1}^n Z_i$ are independent, which is always the case if the observations are independent. 
This slight generalization is important for certain examples with dependent observations that lead to the LLR with independent increments. 
See, e.g., Example~\ref{Ex1} in Section~\ref{sec:Examples}.

\begin{theorem}\label{Th:FOAOindep}
Let $T_A$ be the Shiryaev changepoint detection procedure defined in \eqref{ShirProc}.  Let $r\ge 1$. Assume that the LLR process $\{\lambda_{k+n}^k\}_{n\ge 1}$ has independent, 
not necessarily identically distributed increments under $\Pb_k$, $k\in\Zbb_+$.  Suppose that condition \eqref{Pmaxgen} holds and the following condition
\begin{equation}\label{ProbLeftgen}
\lim_{n\to\infty} \Pb_k\brc{\frac{1}{n}\lambda_{\ell+n}^\ell < I-\varepsilon} =0 \quad \text{for all}~ \varepsilon >0, ~ \text{all}~ \ell \ge k~ \text{and all}~ k \in \Zbb_+
\end{equation}  
is satisfied.
 
\noindent {\rm \bf (i)} Let the prior distribution of the change point be geometric \eqref{Geom} with $q=0$.  Then relations \eqref{MADDAOgen1}, \eqref{FOAOmomentsgen} and \eqref{MADDAOgen} 
hold true for all $m >0$ with $\mu=|\log(1-\rho)|$. Therefore, the Shiryaev procedure $T_{A_\alpha}$ minimizes asymptotically 
as $\alpha\to0$ all positive moments of the detection delay in class $\class$.

\noindent {\rm \bf (ii)} Let the prior distribution be geometric with the parameter $\rho=\rho_\alpha$ that depends on $\alpha$ and goes to zero as $\alpha\to 0$ at such rate that
\begin{equation}\label{rate}
\lim_{\alpha\to0} \frac{\log \rho_\alpha}{\log \alpha} = 0.
\end{equation}
Then relations \eqref{MADDAOgen2} hold for all $m>0$, i.e., the Shiryaev procedure is asymptotically optimal with respect to all positive moments of the detection delay.
\end{theorem}

The idea of relaxing the $r$-complete convergence condition by condition \eqref{ProbLeftgen} is based on splitting integration, when obtaining the upper bound for the expectation $\Eb_k[(T_A-k)^+]^r$,
 into a sequence of intervals (cycles) of the size $N_A\approx \log A /(I+\mu)$ and then showing that $\Pb_k (T_A-k > \ell N_A) \le \delta^\ell$, $\ell=1,2,\dots$ 
 for some small $\delta$ under condition \eqref{ProbLeftgen},  using independence of the LLR increments. The details are given below.

\begin{proof}
(i) Hereafter $\lfloor x\rfloor$ denotes the largest integer less or equal to $x$.  Let $N_A=1+\lfloor \log (A/\rho) /(I+\mu-\varepsilon) \rfloor$, where $\mu=|\log(1-\rho)|$.  
We need only to prove that the upper bound \eqref{UppergenA} holds under condition \eqref{ProbLeftgen}. To this end, note that we have the following chain of equalities and inequalities:
\begin{align} \label{ExpkTAplus}
\Eb_{k}\brcs{(T_A-k)^+}^r  & = \sum_{\ell=0}^{\infty} \int_{\ell N_A}^{(\ell+1) N_A} r t^{r-1}  \Pb_k (T_A-k > t) \, \D t    \nonumber
\\
& \le N_A^r + \sum_{\ell =1}^{\infty} \int_{\ell N_A}^{(\ell +1) N_A} r t^{r-1}  \Pb_k (T_A-k > t) \, \D t  \nonumber
\\
& \le  N_A^r + \sum_{\ell=1}^{\infty} \int_{\ell N_A}^{(\ell+1) N_A} r t^{r-1}  \Pb_k (T_A-k > \ell N_A) \, \D t  \nonumber
\\
& = N_A^r\brc{1 + \sum_{\ell=1}^{\infty} [(\ell+1)^r-\ell^r]  \Pb_k (T_A-k > \ell N_A )} \nonumber
\\
 & \le  N_A^{r} \brc{1+\sum_{\ell=1}^{\infty}   r (\ell+1)^{r-1}   \Pb_k (T_A -k>  \ell N_A )} \nonumber
 \\
 & \le  N_A^{r} \brc{1+ r 2^{r-1} \sum_{\ell=1}^{\infty}  \ell^{r-1}   \Pb_k (T_A -k>  \ell N_A )}.  
\end{align}
Introduce the following notation: $A_\rho=A/\rho$, $a_\rho=\log A_\rho - N_A \, |\log(1-\rho)|$, 
\[
R_{n,\rho} = \sum_{m=0}^{n-1} (1-\rho)^{m-n} \exp\set{\lambda_n^m}, \quad R_{n,\rho}^{j} = \sum_{m=j}^{n-1} (1-\rho)^{m-n} \exp\set{\lambda_n^m}, ~~ n > j, ~j=0,1,\dots.
\]
Note that $R_{n,\rho}= \Lambda_n/\rho$ (see \eqref{LambdaGeom}). Since $R_{n,\rho} \ge R_{n,\rho}^{j}\ge (1-\rho)^{j-n}\exp\set{\lambda_n^j}$ (for any $n>j$) and the increments of $\lambda_n^j$ 
are independent, we obtain
\begin{equation}\label{Needit}
\begin{aligned}
\Pb_k \brc{T_A-k > \ell N_A} & =\Pb_k\brc{R_{n,\rho} < A_\rho ~\text{for}~ n=1,\dots,k+\ell N_A}
\\
& \le  \Pb_k\brc{R_{k+n N_A,\rho} < A_\rho ~\text{for}~ n=1,\dots,\ell}
\\
& \le  \Pb_k\brc{R_{k+n N_A,\rho}^{k+(n-1)N_A+1} <A_\rho ~\text{for}~ n=1,\dots,\ell}
\\
& \le \Pb_k\brc{\exp\set{\lambda_{k+ n N_A}^{k+(n-1) N_A+1}} <A_\rho (1-\rho)^{N_A} ~\text{for}~ n=1,\dots,\ell}
\\
& =  \Pb_k\brc{\lambda_{k+N_A}^k< a_\rho, \lambda_{k+2 N_A}^{k+N_A+1}< a_\rho,\dots, \lambda_{k+\ell N_A}^{k+(\ell-1)N_A+1} < a_\rho}
\\
&= \prod_{n=1}^\ell \Pb_k\brc{\lambda_{k+n N_A}^{k+(n-1)N_A+1}< a_\rho} \le \prod_{n=1}^\ell \Pb_k\brc{\frac{1}{N_A}\lambda_{k+n N_A}^{k+(n-1)N_A+1} <I-\varepsilon}  ,
\end{aligned}
\end{equation}
where the last inequality follows from the inequality
\[
\begin{aligned}
\Pb_k\brc{\lambda_{k+n N_A}^{k+(n-1)N_A+1}< a_\rho} &= \Pb_k\brc{\frac{\lambda_{k+n N_A}^{k+(n-1)N_A+1}}{N_A} < \frac{\log A_\rho}{N_A}  - \mu}
\\
& \le \Pb_k\brc{\frac{\lambda_{k+n N_A}^{k+(n-1)N_A+1}}{N_A} <\frac{\log A_\rho}{1+\log A_\rho} (I+\mu-\varepsilon) -\mu} 
\\
&\le \Pb_k\brc{\frac{\lambda_{k+n N_A}^{k+(n-1)N_A+1}}{N_A} <I  -\varepsilon} ,  
\end{aligned}
\]
which holds for all $0<\varepsilon < I+\mu$ and $k\in\Zbb_+$. By condition \eqref{ProbLeftgen}, for a sufficiently large $A$ there exists a small $\delta_A$ such that 
\[
\Pb_k\brc{\frac{1}{N_A}\lambda_{k+n N_A}^{k+(n-1)N_A+1} <I-\varepsilon} \le \delta_A, \quad n \ge 1.
\]
Therefore, for any $\ell \ge 1$, 
\[
\Pb_k \brc{T_A-k > \ell N_A}  \le \delta_A^\ell.
\]
Combining this inequality with \eqref{ExpkTAplus} and using the fact that $L_{r,A}=\sum_{\ell=1}^\infty \ell^{r-1} \delta_A^\ell \to 0$ as $A\to\infty$ for any $r>0$, we obtain
\begin{equation}\label{Momineqrho}
\begin{aligned}
\Eb^\pi[(T_{A}-\nu)^r | T_{A}  > \nu] & = \frac{\sum_{k=0}^\infty \pi_k  \Eb_{k}\brcs{(T_A-k)^+}^r}{1-\PFA(T_A)} 
\\
& \le \frac{\brc{1+\frac{\log (A/\rho)}{I+\mu-\varepsilon}}^r + r 2^{r-1}\, L_{r,A}}{A/(1+A)}   
\\
& = \brc{\frac{\log A}{I+\mu-\varepsilon}}^r(1+o(1))  \quad \text{as}~ A\to\infty.
\end{aligned}
\end{equation}
Since $\varepsilon\in(0,I+\mu)$ is an arbitrary number this implies the upper bound \eqref{UppergenA}.

Applying \eqref{UppergenA} together with the lower bound \eqref{LBTA1} (which holds as before due to condition \eqref{Pmaxgen}) yields \eqref{MADDAOgen1}.

Next, under condition \eqref{Pmaxgen}, for all $m>0$, we have the asymptotic lower bound \eqref{LBinclass} in class $\class$. Substituting $\log A_\alpha \sim |\log\alpha|$ 
in \eqref{MADDAOgen1} (in particular, we may take $A_\alpha=(1-\alpha)/\alpha$) we immediately obtain the asymptotic approximation   \eqref{MOapproxgenalpha} for moments of the detection delay 
of the Shiryaev procedure $T_{A_\alpha}$. This proves \eqref{FOAOmomentsgen}. 
Asymptotic approximations \eqref{MADDAOgen} are obvious from \eqref{MOapproxgenalpha} and \eqref{FOAOmomentsgen}.  This completes the proof of (i).

(ii) Substituting $\rho=\rho_\alpha$ and $A=A_\alpha=(1-\alpha)/\alpha$ (or more generally $\log A_\alpha \sim |\log \alpha|$) in inequality \eqref{Momineqrho}, we obtain
\begin{equation*}
\Eb^{\pi^\alpha}[(T_{A_\alpha}-\nu)^r | T_{A_\alpha}  > \nu]  \le \frac{\brc{1+\frac{\log ((1-\alpha)/\alpha) + |\log \rho_\alpha|}{I+\mu_\alpha-\varepsilon}}^r + 
r 2^{r-1}\, L_{r,A_\alpha}}{1-\alpha} .
\end{equation*} 
By condition \eqref{rate}, the right side is asymptotically as $\alpha\to0$ equal to 
\[
\brc{\frac{|\log \alpha|}{I-\varepsilon}}^r(1+o(1)),
\]
which along with the lower bound \eqref{LBinclass} ($\mu_\alpha\to0$ as $\alpha\to0$) completes the proof.
\end{proof}

\begin{remark}
The assertions of Theorem~\ref{Th:FOAOindep} hold whenever the normalized LLR processes $n^{-1}\lambda_{\ell+n}^\ell$, $\ell =k, k+1,\dots$ converge almost surely to a constant $I$ 
under $\Pb_k$ for all $k \in \Zbb_+$, 
since in this case both conditions \eqref{Pmaxgen} and \eqref{ProbLeftgen} are satisfied. In the iid case, the assertions of the 
theorem are true with $I=\mc{K}$ being the Kullback--Leibler information number \eqref{KL},
assuming that $0<\mc{K}<\infty$. Indeed, in this case, the conditions of Theorem~\ref{Th:FOAOindep} hold with $I=\mc{K}$ by the SLLN. This result has been previously established by \cite{TV2005}
using a completely different technique.
\end{remark}

\begin{remark}
Theorem~\ref{Th:FOAOindep}(i) can be generalized for the arbitrary prior distribution satisfying condition ($ \Cb$) and Theorem~\ref{Th:FOAOindep}(ii)  for prior distributions satisfying condition \eqref{Prior} 
with parameter $\mu=\mu_\alpha=o(|\log\alpha|)$ as $\alpha\to0$. However, in this general case, the proof becomes very tedious and obscures the main ideas. For this reason, we focused on 
the geometric prior, which is not an overly restrictive assumption, especially in part (ii).  
\end{remark}

\section{Asymptotic operating characteristics of the Shiryaev--Roberts procedure}\label{sec:SR} 
 
 In this section, we discuss asymptotic operating characteristics of the SR procedure $\wtT_B$ defined in \eqref{SRproc} and \eqref{SRstat}. While the methods are similar to those used in the previous section,
 there are specific features and certain differences that have to be considered separately.

 \subsection{The non-iid case}\label{ssec:NoniidcaseSR}
 
 As mentioned in Subsection~\ref{ssec:Procedures}, the SR statistic $R_n$ 
 is the limit of the statistic $\Lambda_n/\rho$ as $\rho\to0$ when the prior distribution is geometric \eqref{Geom} with $q=0$.
 Therefore, it is intuitively appealing, based on the results of Theorem~\ref{Th:FOAOgen2}, that 
 \[
 \Eb^\pi(\wtT_B -\nu)^r | \wtT_B>\nu] \sim \brc{\frac{\log B}{I}}^r \quad \text{as}~ B\to \infty,
 \]
and therefore, if we can select $B=B_\alpha$ so that $\PFA(\wtT_{B_\alpha}) \le \alpha$ and $\log B_\alpha \sim  |\log\alpha|$, then the SR procedure is also asymptotically as $\alpha\to 0$ 
optimal whenever $\rho_\alpha\to 0$ at an appropriate rate. Below we show that this is indeed true.

The first question is how to select threshold $B_\alpha$ in order to embed the SR procedure into class $\class$. To answer this question, it suffices to note that under $\Pb_\infty$ the 
SR statistic $R_n$ is a  submartingale with mean $\Eb_\infty R_n=n$, so that applying Doob's submartingale inequality, we obtain
 \[
 \Pb_\infty(\wtT_B \le j) =\Pb_\infty\brc{\max_{1\le i \le j} R_i \ge B} \le j/B, \quad j=1,2,\dots,
 \]
and hence,
\begin{equation}\label{PFASR}
\PFA(\wtT_{B}) =\sum_{j=0}^\infty \pi_j \Pb_\infty(\wtT_B \le j) \le \bar{\nu} /B ,
\end{equation}
where $\bar{\nu} = \sum_{j=1}^\infty j \pi_j$. Therefore, assuming that $\bar{\nu}<\infty$, we obtain that setting $B=B_\alpha = \bar\nu/\alpha$ implies $\wtT_{B_\alpha} \in \class$. 
If, in a particular case, the prior distribution is geometric, then $\PFA(\wtT_{B}) \le (1-\rho)/(\rho B)$. 

\begin{theorem}\label{Th:AOCSR}
Let $\wtT_B$ be the SR changepoint detection procedure defined in  \eqref{SRproc}. Let $\bar\nu=\sum_{j=1}^\infty j \pi_j <\infty$. Let $r\ge 1$. 
Assume that for some  number $0<I<\infty$ conditions \eqref{Pmaxgen} and \eqref{rcompleteleftgen} are satisfied.

\noindent {\rm \bf (i)}  Then for all $0<m \le r$
\begin{equation} \label{MomentsSR1}
\lim_{B\to\infty}  \frac{\Eb^\pi[(\wtT_B-\nu)^m | \wtT_B > \nu]}{(\log B)^m} =\frac{1}{I^m}.
\end{equation}

\noindent {\rm \bf (ii)}  Let $B=B_\alpha=\bar\nu/\alpha$. Then $\wtT_{B_\alpha}\in \class$ and for all $0<m \le r$,
\begin{equation}\label{MomentsSR2}
\lim_{\alpha\to0} \frac{\Eb^\pi[(\wtT_{B_\alpha}-\nu)^m | \wtT_{B_\alpha} > \nu]}{|\log \alpha|^m} =\frac{1}{I^m} .
\end{equation}
This assertion also holds if $B=B_\alpha$ is selected so that $\PFA(T_{B_\alpha}) \le \alpha$ and $\log B_\alpha\sim |\log\alpha|$ as $\alpha\to0$.
\end{theorem}

\begin{proof}
(i)  For $\varepsilon\in(0,1)$, let $M_{B,\varepsilon} = (1-\varepsilon) I^{-1} \log B$. Using Chebyshev's inequality and inequality \eqref{PFASR},  similarly to \eqref{A1} we obtain that
\begin{equation}\label{LBNA}
\Eb^\pi[(\wtT_B-\nu)^m | \wtT_B >\nu] \ge M_{B,\varepsilon}^m\brcs{1-\bar\nu/B - \Pb^\pi \brc{0 < \wtT_B-\nu < M_{B,\varepsilon}}} .
\end{equation}
Now, similarly to \eqref{Pktauupper},
\begin{equation}\label{PkTBupper}
 \Pb_k\brc{0 < \wtT_B -k < M_{B,\varepsilon}} \le  U_{B,\varepsilon}^{k}(\wtT_B)  + \beta_{B,\varepsilon}^{k} ,
\end{equation}
where
\[
\begin{aligned} 
U_{B,\varepsilon}^{k}(\wtT_B) = e^{(1+\varepsilon) I M_{B,\varepsilon}}\Pb_\infty\brc{0 < \wtT_B - k <M_{B,\varepsilon}}, \quad
\beta_{B,\varepsilon}^{k}  = \Pb_k\brc{\frac{1}{M_{B,\varepsilon}} \max_{1\le n \le M_{B,\varepsilon}} \lambda_{k+n}^k \ge (1+\varepsilon) \, I} .
\end{aligned}
\]
 Since 
\[
\Pb_\infty\brc{0 < \wtT_B - k <M_{B,\varepsilon}} \le \Pb_\infty\brc{\wtT_B  < k+ M_{B,\varepsilon}} \le (k+M_{B,\varepsilon})/B,
\]
we have
\begin{equation}\label{UpperU}
 U_{B,\varepsilon}^{k}(\wtT_B) \le \frac{k+ (1-\varepsilon) I^{-1} \log B}{B^{\varepsilon^2}}.
\end{equation}
Let $K_B$ be an integer number that approaches infinity as $B\to\infty$. Using \eqref{PkTBupper} and \eqref{UpperU}, we obtain the following upper bound 
\begin{align}\label{ProbineqSR}
\Pb^\pi(0< \wtT_B -\nu < M_{B,\varepsilon}) & = \sum_{k=0}^\infty  \pi_k  \Pb_k\brc{0 < \wtT_B -k < M_{B,\varepsilon}} \nonumber
 \\
 &\le \Pb(\nu > K_{B}) +   \sum_{k=0}^\infty \pi_k U_{B,\varepsilon}^{k}(\wtT_B) + \sum_{k=0}^{K_{B}}  \pi_k  \beta_{B,\varepsilon}^{k} \nonumber
 \\ 
 & \le  \Pb(\nu > K_{B}) + \frac{\bar\nu+(1-\varepsilon) I^{-1} \log B}{B^{\varepsilon^2}}+  \sum_{k=0}^{K_{B}}  \pi_k  \beta_{B,\varepsilon}^{k},
 \end{align}
where the first two  terms go to zero as $B\to\infty$ for all $\varepsilon>0$ since $\bar\nu$ is finite (note that by Markov's inequality $\Pb(\nu>K_B) \le \bar\nu/K_B$) and the last term also goes to zero  
by condition \eqref{Pmaxgen} and Lebesgue's dominated convergence theorem. Thus, for all $0<\varepsilon <1$,
\[
\Pb^\pi(0< \wtT_B -\nu < M_{B,\varepsilon})  \to 0 \quad \text{as}~ B\to\infty
\]   
and applying inequality \eqref{LBNA}, we obtain that for any $0<\varepsilon <1$ as $B\to\infty$
\[
\Eb^\pi[(\wtT_B-\nu)^m | \wtT_B >\nu] \ge (1-\varepsilon)^m \brc{\frac{\log B}{I}}^m (1+o(1)).
\]
Since $\varepsilon$ can be arbitrarily small, this inequality yields the asymptotic lower bound (for any $m>0$)
\begin{equation}\label{LBA}
\Eb^\pi[(\wtT_B-\nu)^m | \wtT_B >\nu] \ge \brc{\frac{\log B}{I}}^m (1+o(1)) \quad \text{as}~ B\to\infty .
\end{equation}

To complete the proof of assertion (i), we now need to show that this lower bound is also an upper bound asymptotically as $B\to\infty$. In just the same way as in the proof of Lemma~\ref{Lem:UpperEk}
(see \eqref{Ektauineq}), we obtain
\begin{align}\label{EkTBineq}
\Eb_{k}\brcs{(\wtT_B-k)^+}^r  & = \int_0^\infty r t^{r-1} \Pb_k\brc{\wtT_B -k > t} \, \D t   \nonumber
 \\
 & \le N_{B,\varepsilon}^{r} + r 2^{r-1} \sum_{n=N_{B,\varepsilon}}^{\infty}   n^{r-1}   \Pb_k\brc{\wtT_B -k >  n},
 \end{align}
where $N_{B,\varepsilon}=1+\lfloor (\log B)/(I-\varepsilon) \rfloor$.
Clearly, $R_n \ge e^{\lambda_n^k}$ (for any $n>k$), and therefore,
\[
 \Pb_k\brc{\wtT_B -k >  n} = \Pb_k \brc{\max_{1\le i \le n+k} R_i < B} \le \Pb_k \brc{\frac{1}{n} \lambda_{k+n}^k < \frac{1}{n} \log B} .
\]
But  for all $k \in \Zbb_+$ and $n\ge N_{B,\varepsilon}$ the latter probability can be upper-bounded as
\[
\Pb_k \brc{\frac{1}{n} \lambda_{k+n}^k < \frac{1}{n} \log B} \le \Pb_k \brc{\frac{1}{n} \lambda_{k+n}^k < I-\varepsilon},
\]
so that for all $k \in \Zbb_+$ and $n\ge N_{B,\varepsilon}$
\[
 \Pb_k\brc{\wtT_B -k >  n} \le \Pb_k \brc{\frac{1}{n} \lambda_{k+n}^k < I-\varepsilon}.
\]
Substituting this upper bound in inequality \eqref{EkTBineq} yields (for every $0<\varepsilon < I$)
\begin{equation*}
\begin{aligned}
\Eb^\pi[(\wtT_B-\nu)^r | \wtT_B > \nu] & = \frac{\sum_{k=0}^\infty \pi_k  \Eb_{k}[(\wtT_B-k)^+]^r}{1-\PFA(\wtT_B)} 
\\
&\le\frac{\brc{1+\frac{\log B}{I-\varepsilon}}^r + r 2^{r-1}\, \sum_{k=0}^\infty \pi_k  \Upsilon_{k,r}(\varepsilon)}{1-\bar\nu/B} ,
\end{aligned}
\end{equation*}
where we used the inequality $1-\PFA(\wtT_B) \ge 1-\bar\nu/B$. By condition \eqref{rcompleteleftgen},
\[
 \sum_{k=0}^\infty \pi_k  \Upsilon_{k,r}(\varepsilon) < \infty \quad \text{for any}~ \varepsilon >0,
\]
which implies that, for every $0<\varepsilon < I$
\[
\Eb^\pi[(\wtT_B-\nu)^r | \wtT_B > \nu] \le \brc{\frac{\log B}{I-\varepsilon}}^r (1+o(1)) \quad \text{as}~ B \to \infty.
\]
Since $\varepsilon$ can be arbitrarily small, this implies the asymptotic upper bound
\begin{equation}\label{UBB}
\Eb^\pi[(\wtT_B-\nu)^r | \wtT_B > \nu] \le \brc{\frac{\log B}{I}}^r (1+o(1)) \quad \text{as}~ B \to \infty.
\end{equation}
This completes the proof of (i).

(ii) Substitution of $B=B_\alpha=\bar\nu/\alpha$ (or more generally $\log B_\alpha \sim |\log \alpha|$) in \eqref{MomentsSR1} immediately yields \eqref{MomentsSR2}, and the proof is complete.
\end{proof}

Theorem~\ref{Th:AOCSR} does not cover the case of prior distributions with exponential tails ($\mu>0$) but with $\mu=\mu_\alpha$ that depends on the PFA $\alpha$ and approaches zero as $\alpha\to0$. 
The next theorem, which is similar to Theorem~\ref{Th:FOAOgen2}, addresses this case.

\begin{theorem}\label{Th:AoptSR}
Let $r\ge 1$.  Assume that the prior distribution $\pi^\alpha=\{\pi_k^\alpha\}$ of the change point $\nu$ satisfies condition \eqref{Prior} 
with $\mu=\mu_\alpha\to 0$ as $\alpha\to 0$ and that $\mu_\alpha$
approaches zero at such rate that $\bar\nu_\alpha=\sum_{k=1}^\infty  k \, \pi_k^\alpha$ increases at a rate slower than $|\log\alpha|$ as $\alpha\to0$, i.e., 
\begin{equation}\label{Prior4}
\lim_{\alpha\to 0} \frac{{\bar\nu_\alpha}}{|\log \alpha|} = 0.
\end{equation} 
Assume that for some $0<I<\infty$ conditions \eqref{Pmaxgen} and \eqref{uniformrcomplete} are satisfied.  
If $B=B_\alpha=\bar\nu_\alpha/\alpha$, then $\PFA(T_{B_\alpha}) \le \alpha$ and  for all $0<m \le r$
\begin{equation} \label{MomentsSRmu0}
\inf_{T \in \class} \Eb^{\pi^\alpha}[(T-\nu)^m | T> \nu] \sim  \Eb^{\pi^\alpha} [(\wtT_{B_\alpha}-\nu)^m | \wtT_{B _\alpha}> \nu] \sim \brc{\frac{|\log\alpha|}{I}}^m \quad \text{as}~ \alpha \to 0.
\end{equation}
Therefore, the SR procedure $\wtT_{B_\alpha}$ is asymptotically optimal as $\alpha\to0$ in class $\class$, minimizing moments of the detection delay up to order $r$.
\end{theorem}

\begin{proof}
Similarly to \eqref{LBNA}, 
\begin{equation*}
\Eb^{\pi^\alpha}[(\wtT_{B_\alpha}-\nu)^m | \wtT_{B_\alpha} >\nu] \ge M_{B_\alpha,\varepsilon}^m\brcs{1-\alpha - \Pb^{\pi^\alpha} \brc{0 < \wtT_{B_\alpha}-\nu < M_{B_\alpha,\varepsilon}}} .
\end{equation*}
Let $K_\alpha= \lfloor C |\log\alpha|\rfloor$ with some positive constant $C$. Substituting $B=B_\alpha=\bar\nu_\alpha/\alpha$ in inequality \eqref{ProbineqSR}, we obtain
\[
\Pb^{\pi^\alpha}(0< \wtT_{B_\alpha} -\nu < M_{B_\alpha,\varepsilon})  \le  \Pb(\nu > K_{\alpha}) + \frac{\bar\nu_\alpha+(1-\varepsilon) I^{-1} \log (\bar\nu_\alpha/\alpha)}{(\bar\nu_\alpha/\alpha)^{\varepsilon^2}}+  
\sum_{k=0}^{K_{\alpha}}  \pi_k^\alpha  \beta_{B_\alpha,\varepsilon}^{k}.
 \] 
As before, the last term approaches zero as $\alpha\to0$. It is easily verified that the middle term also approaches zero as long as condition \eqref{Prior4} is satisfied. Finally, by the Markov inequality 
and condition \eqref{Prior4},
\[
\Pb(\nu > K_{\alpha}) \le \bar\nu_\alpha/K_\alpha =o(|\log\alpha|)/C |\log\alpha| \to 0 \quad \text{as}~\alpha\to0,
\]
so $\lim_{\alpha\to0} \Pb(\nu > K_{\alpha})=0$. Therefore, we conclude that for all $0<\varepsilon <1$,
\[
\Pb^\pi(0< \wtT_{B_\alpha} -\nu < M_{B_\alpha,\varepsilon})  \to 0 \quad \text{as}~ \alpha\to 0 .
\]   
Since by \eqref{Prior4}, $\log B_\alpha\sim |\log\alpha|$, we obtain that, for all $m>0$,
\begin{equation}\label{LBalpha}
\Eb^{\pi^\alpha}[(\wtT_{B_\alpha}-\nu)^m | \wtT_{B_\alpha} >\nu] \ge \brc{\frac{|\log \alpha|}{I}}^m (1+o(1)) \quad \text{as}~ \alpha\to0 .
\end{equation}

The upper bound
\begin{equation}\label{UBalpha}
\Eb^{\pi^\alpha}[(\wtT_{B_\alpha}-\nu)^m | \wtT_{B_\alpha} >\nu] \le \brc{\frac{|\log \alpha|}{I}}^m (1+o(1)) \quad \text{as}~ \alpha\to0 
\end{equation}
is obtained in the manner absolutely analogous to the proof of the upper bound \eqref{UBB} in the previous theorem. Specifically, 
for all $k \in \Zbb_+$ and $n\ge N_{B_\alpha,\varepsilon}$
\[
 \Pb_k\brc{\wtT_{B_\alpha} -k >  n} \le \Pb_k \brc{\frac{1}{n} \lambda_{k+n}^k < I-\varepsilon}
\]
and 
\[
\Eb_{k}\brcs{(\wtT_B-k)^+}^r   \le N_{B_\alpha,\varepsilon}^{r} + r 2^{r-1} \sum_{n=N_{B_\alpha,\varepsilon}}^{\infty}   n^{r-1}   \Pb_k\brc{\wtT_{B_\alpha} -k >  n},
 \]
so that 
\begin{equation*}
\Eb^{\pi^\alpha}[(\wtT_{B_\alpha}-\nu)^r | \wtT_{B_\alpha} > \nu] \le\frac{\brc{1+\frac{\log(\bar\nu_\alpha/\alpha)}{I-\varepsilon}}^r + r 2^{r-1}\, \sup_{k\ge0} \Upsilon_{k,r}(\varepsilon)}{1-\alpha}.
\end{equation*}
Taking into account that, by condition \eqref{Prior4}, $\log(\bar\nu_\alpha/\alpha)\sim|\log\alpha|$ and that, by condition \eqref{uniformrcomplete}, $\sup_{k\ge0}\Upsilon_{k,r}(\varepsilon)<\infty$,
yields the asymptotic upper bound \eqref{UBalpha}.

Now, applying the bounds \eqref{LBalpha} and \eqref{UBalpha} simultaneously, we obtain the asymptotic approximation
\[
\Eb^{\pi^\alpha}[(\wtT_{B_\alpha}-\nu)^m | \wtT_{B _\alpha}> \nu] \sim \brc{\frac{|\log\alpha|}{I}}^m \quad \text{as}~ \alpha \to 0,
\]
i.e., the second approximation in \eqref{MomentsSRmu0}. The first one follows from the lower bound \eqref{LBinclass} (with $\mu=\mu_\alpha \to 0$). The proof is complete. 
\end{proof}

\begin{remark}
If the prior distribution is geometric \eqref{Geom}, then $\mu=|\log(1-\rho)|$ and $\bar\nu=(1-\rho)/\rho$ and Theorem~\ref{Th:AoptSR} holds whenever 
the parameter $\rho=\rho_\alpha\to0$ at the rate $|\log\rho_\alpha|=o(|\log\alpha|)$. 
\end{remark}

Comparing asymptotic formula \eqref{MomentsSR2} with asymptotic lower bound \eqref{LBinclass} in class $\class$, we see that, opposed to the Shiryaev procedure, the SR procedure is 
not asymptotically optimal so long as $\mu>0$, i.e., if the prior distribution has exponential tail. If the tail is heavy, i.e., $\mu=0$, then the SR procedure is asymptotically optimal and, 
by Theorem~\ref{Th:AoptSR}, the same is true if
$\mu=\mu_\alpha \to 0$ as $\alpha\to0$ at a suitable rate. This is intuitively expected, since the SR statistic does not exploit the prior distribution, relying on the improper uniform prior. 
However, there still may be a problem when applying the latter asymptotic 
result in practice. Indeed, there is no guarantee that the bound $\bar\nu/B$ in inequality \eqref{PFASR} is relatively tight in a sense that $|\log \PFA(\wtT_B)| \sim \log B$ as $B\to\infty$,
i.e., that for a sufficiently large $B$, $\PFA(\wtT_B) \approx \const/B$, unless $\mu/I$ is small if the prior satisfies condition \eqref{Prior} with $\mu>0$. Even in the case of heavy-tailed priors 
($\mu=0$) this is perhaps not true. In this respect, \cite{TM2010} conjectured that asymptotically as $B\to\infty$ the accurate 
approximation is $\PFA(\wtT_B) \sim O(1)/B^{s(\mu)}$, where
$s(\mu)>1$ for all $\mu>0$ and $s(\mu) \to 0$ as $\mu\to0$. If this conjecture is correct, which is partially justified in \cite{TM2010} by numerical computations, then the asymptotic relative efficiency of the 
asymptotically optimal Shiryaev procedure compared to the SR procedure is $[I s(\mu)/(I+\mu)]^m$, but not $[I/(I+\mu)]^m$, as Theorem~\ref{Th:AOCSR} suggests. Note that this is expected to be true only for
the priors with the exponential tail, but not necessarily for heavy-tailed priors.

 \subsection{The case of independent observations}\label{ssec:IndepcaseSR}

We now provide a theorem for the SR procedure similar to Theorem~\ref{Th:FOAOindep} in the case where  the LLR has independent increments.

\begin{theorem}\label{Th:FOAOindepSR}
Let $\wtT_B$ be the SR changepoint detection procedure. Let $r\ge 1$. Assume that the LLR process $\{\lambda_{k+n}^k\}_{n\ge 1}$ has independent, 
not necessarily identically distributed increments under $\Pb_k$, $k\in\Zbb_+$.  Suppose that conditions \eqref{Pmaxgen} and \eqref{ProbLeftgen} are satisfied.
 
\noindent {\rm \bf (i)} Let the prior distribution $\pi=\{\pi_k\}$ be geometric \eqref{Geom} with $q=0$.  
Then relation \eqref{MomentsSR1} holds for all $m>0$. If $B=B_\alpha=(1-\rho)/\rho\alpha$, then 
$\wtT_{B_\alpha}\in \class$ and relation \eqref{MomentsSR2} holds for all $m>0$. 

\noindent {\rm \bf (ii)} Let the prior distribution $\pi^{\alpha}=\{\pi_k^\alpha\}$ be geometric with the parameter $\rho=\rho_\alpha$ that depends on $\alpha$ and goes 
to zero as $\alpha\to 0$ at such a rate that condition \eqref{rate} is satisfied.
Then relations \eqref{MomentsSRmu0} hold for all $m>0$, i.e., the SR procedure is asymptotically optimal with respect to all positive moments of the detection delay.
\end{theorem}

\begin{proof}
(i)   Again let $N_{B,\varepsilon}=1+\lfloor \log (B) /(I-\varepsilon) \rfloor$.  Similarly to \eqref{ExpkTAplus}, we obtain 
\begin{equation}\label{ExpkTAplusSR}
\Eb_{k}\brcs{(\wtT_B-k)^+}^r   \le  N_{B,\varepsilon}^{r} \brc{1+ r 2^{r-1} \sum_{\ell=1}^{\infty}  \ell^{r-1}   \Pb_k (\wtT_B -k>  \ell N_{B,\varepsilon})}.  
\end{equation}
Since
\[
R_{n} \ge R_{n}^{j} = \sum_{m=j}^{n-1}  e^{\lambda_n^m} \ge e^{\lambda_n^j}, ~~ n > j, ~j=0,1,\dots,
\]
and the increments of $\lambda_n^j$ are independent, as in \eqref{Needit}, we obtain that for all $0<\varepsilon < I$ and $k\in\Zbb_+$
\begin{equation*}
\begin{aligned}
\Pb_k \brc{\wtT_B-k > \ell N_{B,\varepsilon}} \le  \prod_{n=1}^\ell \Pb_k\brc{\lambda_{k+n N_{B,\varepsilon}}^{k+(n-1)N_{B,\varepsilon}+1}< \log B} \le 
\prod_{n=1}^\ell \Pb_k\brc{\frac{1}{N_{B,\varepsilon}}\lambda_{k+n N_{B,\varepsilon}}^{k+(n-1)N_{B,\varepsilon}+1} <I-\varepsilon} .
\end{aligned}
\end{equation*}
By condition \eqref{ProbLeftgen}, for a sufficiently large $B$ there exists a small $\delta_B$ such that 
\[
\Pb_k\brc{\frac{1}{N_{B,\varepsilon}}\lambda_{k+n N_{B,\varepsilon}}^{k+(n-1)N_{B,\varepsilon}+1} <I-\varepsilon} \le \delta_B, \quad n \ge 1,
\]
so that, for any $\ell \ge 1$, $\Pb_k \brc{\wtT_B-k > \ell N_{B,\varepsilon}}  \le \delta_B^\ell$. This along with \eqref{ExpkTAplus} and the fact that 
$L_{r,B}=\sum_{\ell=1}^\infty \ell^{r-1} \delta_B^\ell \to 0$ as $B\to\infty$ for any $r>0$ yields
\begin{equation}\label{IneqSRB}
\begin{aligned}
\Eb^\pi[(\wtT_B-\nu)^r | \wtT_B  > \nu] & \le \frac{\brc{1+\frac{\log B}{I-\varepsilon}}^r + r 2^{r-1}\, L_{r,B}}{1-\bar\nu/B}   
\\
& = \brc{\frac{\log B}{I-\varepsilon}}^r(1+o(1))  \quad \text{as}~ B\to\infty .
\end{aligned}
\end{equation}
Since $\varepsilon$ is arbitrarily small, it follows that
\begin{equation}\label{MomineqrhoSR}
\Eb^\pi[(\wtT_B-\nu)^r | \wtT_B  > \nu] \le \brc{\frac{\log B}{I}}^r(1+o(1))  \quad \text{as}~ B\to\infty.
\end{equation}
Since the lower bound \eqref{LBA}  holds even in a more general case this proves \eqref{MomentsSR1}.

Next, under condition \eqref{Pmaxgen} the asymptotic lower bound \eqref{LBalpha} holds (for all $m>0$) in class $\class$ even in a more general case 
(see the proof of Theorem~\ref{Th:AoptSR}). Substituting in \eqref{MomineqrhoSR} 
$B_\alpha=(1-\rho)/\rho\alpha$ (or more generally $\log B_\alpha\sim|\log \alpha|$), we obtain \eqref{MomentsSR2}, which completes the proof of (i).

(ii) Substituting $B=B_\alpha=(1-\rho_\alpha)/\rho_\alpha\alpha$ (or more generally $\log B_\alpha \sim |\log \alpha|$) in inequality \eqref{MomineqrhoSR}, we obtain
\[
\Eb^{\pi^\alpha}[(\wtT_{B_\alpha}-\nu)^r | \wtT_{B_\alpha}  > \nu] \le  \brc{\frac{|\log \alpha|}{I}}^r(1+o(1)) .
\]
This upper bound along with the lower bound \eqref{LBinclass} with $\mu_\alpha\to0$ as $\alpha\to0$ proves asymptotic relations \eqref{MomentsSRmu0}. 
\end{proof}

\begin{remark}
Both conditions \eqref{Pmaxgen} and \eqref{ProbLeftgen} are satisfied when the normalized LLRs $n^{-1}\lambda_{\ell+n}^\ell$, $\ell =k, k+1,\dots$ converge a.s.\ to $I$ under $\Pb_k$. 
Therefore, the SLLN is sufficient for Theorem~\ref{Th:FOAOindepSR}, i.e., for asymptotic optimality of the SR procedure with respect to all positive moments of the detection delay.
\end{remark}

\section{Examples}\label{sec:Examples}

We now consider three examples that illustrate the general asymptotic theory developed in previous sections.

\begin{example}[Detection of a deterministic signal in AR noise]\label{Ex1}

Let $S_n$ be a deterministic function (signal) that appears at an unknown time $\nu$ in additive noise $\xi_n$, so the observations have the standard ``signal-plus-noise/clutter'' form  
$$
X_{n}=\Ind{n > \nu} S_n +\xi_{n}\,,\quad n \ge 1,
$$
where $\{\xi_n\}_{n \in\Zbb_+}$ is a $p$-th order autoregression (AR$(p)$ process) driven by the normal  $\Nc(0,\sigma^2)$ iid sequence $\{w_n\}_{n\ge 1}$, i.e., 
the sequence $\{\xi_n\}_{n\ge 1}$ obeys the recursion
\begin{equation}\label{sec:Ex.1}
\xi_{n} = \sum_{i=1}^p \beta_i \xi_{n-i} + w_{n}, \quad n \ge 1, \quad \xi_{1-p}=\xi_{2-p}=\cdots=\xi_0=0.
\end{equation}
The coefficients $\beta_1,\dots,\beta_p$ and variance $\sigma^2$ are known, and we suppose that $\beta_1+\cdots+\beta_p \neq 1$. Let $\varphi(x)=(2\pi)^{-1/2}\,e^{-x^2/2}$ 
denote density of the standard normal distribution. Define the $p$-th order residual 
\[
\wtX_n = \begin{cases}
X_n- \sum_{i=1}^p \beta_i X_{n-i} & ~\text{for}~ n > p
\\
X_n- \sum_{i=1}^{j-1} \beta_i X_{j-i} &~ \text{for}~ 1 \le n=j \le p
\end{cases} .
\]
 It is easy to see that pre- and post-change conditional densities  $p_{\infty}(X_{n}|\Xb^{n-1})$ and $p_{0}(X_{n}|\Xb^{n-1})$
are
\begin{equation}\label{DensitiesEx1}
p_{\infty}(X_{n}|\Xb^{n-1})=  \frac{1}{\sigma} \varphi\brc{\frac{\wtX_{n}}{\sigma}}, \quad p_{0}(X_{n}|\Xb^{n-1}) = \frac{1}{\sigma}\varphi\brc{\frac{\wtX_{n}-\wtS_{n}}{\sigma}} ,
\end{equation}
where 
\[
\wtS_n= S_n - \sum_{i=1}^p \beta_i S_{n-i} \quad \text{for}~ n >p \quad \text{and} \quad \wtS_n= S_n - \sum_{i=1}^{j-1} \beta_i S_{j-i} \quad \text{for}~ 1 \le n=j \le p.
\]
Using \eqref{DensitiesEx1}, we obtain that for all $k \in \Zbb_+$ and $n \ge 1$  the LLR has the form
$$
\lambda_{k+n}^{k} = \frac{1}{\sigma^2} \sum_{j=k+1}^{k+n} \wtS_{j} \wtX_{j} -\frac{1}{2\sigma^2} \sum_{j=k+1}^{k+n} \wtS_j^2.
$$
Thus, the initial problem of detection of the signal $S_n$ that appears at unknown time $\nu$ in the correlated AR noise reduces 
to detection of the transformed signal $\wtS_n$ in white Gaussian noise. As a result, the LLR has independent (but not identically distributed) increments.
Since under measure $\Pb_{k}$ the random variables $\{\wtX_{n}\}_{n\ge k+1}$ are independent Gaussian random variables with mean $\Eb_{k}\wtX_n=\wtS_n$ and variance $\sigma^{2}$, under 
$\Pb_k$ the LLR can be represented as 
$$
\lambda_{k+n}^{k}= \frac{1}{2\sigma^2} \sum_{j=k+1}^{k+n} \wtS_j^2 + \frac{1}{\sigma}\sum_{j=k+1}^{k+n} \wtS_j \eta_j,
$$
where $\eta_j$, $j\ge k+1$ are iid standard normal random variables, $\eta_j \sim \Nc(0,1)$. 

Assume that the energy of the transformed signal is an asymptotically linear function, i.e., 
\[
\lim_{n\to\infty} \frac{1}{n} \sum_{j=k+1}^{k+n} \wtS_j^2 = \wtS^2, \quad 0<\wtS^2 <\infty.
\]
Then for all $k\in\Zbb_+$
\[
\frac{1}{n} \lambda_{k+n}^k \xra[n\to\infty]{\Pb_k-\text{a.s.}} \frac{\wtS^2}{2\sigma^2} =I
\]
and, by Theorem~\ref{Th:FOAOindep}, the Shiryaev procedure minimizes as $\alpha\to0$ all positive moments of the detection delay for any value of the parameter $0<\rho<1$ of the geometric prior. 
By Theorem~\ref{Th:FOAOindepSR}, the SR procedure also minimizes all moments of the detection delay
if $\rho=\rho_\alpha \to 0$ and $|\log \rho_\alpha| = o(|\log\alpha|)$ as $\alpha\to0$. If $S_n=S$ does not depend on $n$, then
\[
I = \frac{S^2}{2\sigma^2}\brc{1- \sum_{i=1}^p \beta_i}^2.
\]
\end{example}

\begin{example}[Detection of a change of variance in normal population with unknown mean] \label{Ex2}

Let observations $X_n\sim \Nc(\theta, \sigma^2_\infty)$ be iid normal with variance $\sigma^2_\infty$ before change and iid normal $\Nc(\theta, \sigma^2_0)$ with variance $\sigma^2_0$ 
after change with the same unknown mean~$\theta$. 
Formally, this problem is not in the class of problems considered in this paper since pre- and post-change densities depend on an unknown nuisance parameter $\theta$, and hence, 
the hypotheses are not simple, but composite. However, this problem can be reduced to the problem of testing simple hypotheses using the principle of invariance, since it    
is invariant under the group of shifts $ \{G_{b}(x)=x +b\}_{-\infty<b<\infty}$. The maximal invariant is $\Yb^n=(Y_1, \dots, Y_n)$, $n \ge 2$, where $Y_k = X_k-X_1$, $Y_1=0$, 
and we can now consider a transformed
sequence of observations $\{Y_n\}_{n\ge 2}$, which are not iid and not even independent anymore.   Pre- and post-change densities of $\Yb^n$ are equal to
\begin{equation}  \label{pdfMiGaus1}
\begin{aligned}
p_i(\Yb^n) & = \frac{1}{(2\pi \sigma_i^2)^{n/2}} \int_{-\infty}^{\infty} \exp\left\{-\frac{1}{2\sigma_i^2}\sum_{k=1}^n (Y_k + \theta)^2 \right\} \, \D \theta
\\
& =  \frac{1}{(2\pi \sigma_i^2)^{(n-1)/2} \sqrt{n}} \exp\set{-\frac{1}{2\sigma_i^2}\sum_{k=1}^n (Y_k- \overline{Y}_n)^2}, \quad i=\infty, 0,
\end{aligned}
\end{equation}
where $\overline{Y}_n=n^{-1} \sum_{k=1}^n Y_k$.  Define $\overline{X}_n=n^{-1} \sum_{k=1}^n X_k$, $s_n^2 = (n-1)^{-1}\sum_{k=1}^n (X_k- \overline{X}_n)^2$, and $V_n=(n-1)s_n^2 - (n-2) s_{n-1}$.
Using~\eqref{pdfMiGaus1} and noting that $\sum_{k=1}^n (Y_k- \overline{Y}_n)^2=(n-1) s_n^2$, we obtain 
\[
p_i(Y_j|\Yb^{j-1}) = \frac{1}{\sqrt{2\pi \sigma_i^2}} \sqrt{\frac{j-1}{j}} e^{-V_j/2\sigma^2_i}, \quad j \ge 2,
\]
and therefore, the invariant LLR is
\[
\lambda_{k+n}^k = \sum_{j=k+1}^{k+n} \log \frac{p_0(Y_j|\Yb^{j-1})}{p_\infty(Y_j|\Yb^{j-1})} =  \frac{q^2-1}{2 \sigma_0^2} \sum_{j=k+1}^{k+n} V_j - (n-1) \log q , \quad k+n \ge 2 ~~ (\lambda_1^k=0),
\]
where $q=\sigma_0/\sigma_\infty$. Taking into account that $\sum_{j=1}^{k+n} V_j = (k+n-1) s^2_{k+n}$, we have
\[
\sum_{j=k+1}^{k+n} V_j = (k+n-1) s^2_{k+n}- k s^2_{k+1} ,
\]
so the LLR can be written as
\begin{equation}\label{ILLR}
\lambda_{k+n}^k =\frac{q^2-1}{2 \sigma_0^2} S^2_{k,n} - (n-1) \log q , \quad S_{k,n}^2=  (k+n-1) s^2_{k+n}- k s^2_{k+1} .
\end{equation}
Thus, we can now construct the invariant Shiryaev and SR procedures based on the LLRs $\lambda_n^k$, $n\ge 2$ defined in \eqref{ILLR}.

Note first that $S_{k,n}^2/n \to \sigma_0^2$  as $n\to\infty$ almost surely under~$\Pb_k$, so that
\[
n^{-1}\lambda_{k+n}^k \xrightarrow[n\to\infty]{\Pb_{k}-\text{a.s.}} \frac{q^2-1}{2} - \log q =I,
\]
and $I$ is positive for any $q\neq1$ ($q>0$). Thus, condition \eqref{Pmaxgen} holds with $I=(q^2-1)/2 -\log q$ and to apply the results of previous sections it suffices to show that
for some $r\ge 1$ and any $\varepsilon>0$
\begin{equation}\label{rcompleteILLR}
\sum_{n=1}^\infty n^{r-1} \sup_{k\in\Zbb_+} \Pb_k\brc{\frac{1}{n} \lambda_{k+n}^k  < I-\varepsilon} <\infty .
\end{equation}
To this end, note that the statistic $S_{k,n}^2$ can be written as
\[
S_{k,n}^2 = \sum_{i=k+1}^{k+n} (X_i- \overline{X}_{k+n}^{k+1})^2 + k (\overline{X}_{k+n}-\overline{X}_{k})^2 + n (\overline{X}_{k+n}-\overline{X}_{k+n}^{k+1})^2,
\]
where $\overline{X}_{k+n}^{k+1}= n^{-1} \sum_{i=k+1}^{k+n} X_i$. Denoting
\[
k (\overline{X}_{k+n}-\overline{X}_{k})^2 + n (\overline{X}_{k+n}-\overline{X}_{k+n}^{k+1})^2 = W_{k,n},
\]
and  using the fact that $W_{k,n} \ge 0$, we obtain that for some positive $\tilde\varepsilon$
\[
\begin{aligned}
\Pb_k\brc{\frac{1}{n} \lambda_{k+n}^k  < I-\varepsilon} & = \Pb_k\brc{\frac{1}{n} \sum_{i=k+1}^{k+n} (X_i- \overline{X}_{k+n}^{k+1})^2  
< \sigma_0^2- \frac{2\sigma_0^2}{(q^2-1) n} \log q- \frac{1}{n} W_{k,n} - \tilde\varepsilon}
\\
& \le \Pb_k\brc{\frac{1}{n} \sum_{i=k+1}^{k+n} (X_i- \overline{X}_{k+n}^{k+1})^2  < \sigma_0^2-\tilde\varepsilon}
\\
& = \Pb_0\brc{\frac{1}{n} \sum_{i=1}^{n} (X_i- \overline{X}_{n})^2  < \sigma_0^2-\tilde\varepsilon}
\\
& = \Pb_0\brc{\frac{1}{n} (n-1) s_n^2  < \sigma_0^2-\tilde\varepsilon}.
\end{aligned}
\]
Since $(n-1) s_n^2/\sigma_0^2$ has chi-squared distribution with $n-1$ degrees of freedom,
$\Pb_0\brc{(n-1) s_n^2/n  - \sigma_0^2 < -\tilde\varepsilon}$ vanishes exponentially fast as $n\to\infty$ and it follows that for all $\tilde\varepsilon >0$ and all $r\ge 1$
\[
\sum_{n=1}^\infty n^{r-1} \sup_{k\in\Zbb_+} \Pb_k\brc{\frac{1}{n} \lambda_{k+n}^k  < I-\varepsilon} \le \sum_{n=1}^\infty n^{r-1} \Pb_0\brc{\frac{1}{n} (n-1) s_n^2  < \sigma_0^2-\tilde\varepsilon} <\infty  .
\]
This implies \eqref{rcompleteILLR} for all $r\ge 1$.

By Theorem~\ref{Th:FOAOgen} and Theorem~\ref{Th:FOAOgen2}, the Shiryaev detection procedure minimizes asymptotically all positive moments of the detection delay 
$\Eb^\pi[(T_A-\nu)^m|T_A>\nu]$ (for all $m\ge 1$), and the results of Section~\ref{sec:SR} for the SR procedure can also be applied to all positive moments of the detection delay. 
 \end{example}

 \begin{example}[Detection of a change in the correlation coefficient of the AR$(1)$ process] \label{Ex3}
Let the observations represent the Markov Gaussian sequence with the correlation coefficient $\beta_0$ before change and $\beta_1$ after change, i.e.,
\[
X_n = \brc{\beta_0 \Ind{\nu \le n} + \beta_1\Ind{\nu > n}} X_{n-1} + w_n, \quad n\ge1, \quad X_0=0,
\] 
where $w_n\sim\Nc(0,1)$, $n\ge 1$ are iid standard normal random variables. 
Let $|\beta_i| <1$, $i=0,1$, so that the AR$(1)$ process is stable. The pre- and post-change conditional densities  are 
\[
p_\infty(X_n|\Xb_{n-1}) = \varphi(X_n-\beta_0 X_{n-1}) \quad \text{and} \quad p_0(X_n  | \Xb_{n-1}) = \varphi(X_n-\beta_1 X_{n-1}) .
\]
The stationary distribution $\G(x)=\Pb(X_\infty \le x)$ of the Markov process $\{X_n\}_{n > k}$ under $\Pb_k$ is  given by the random variable 
$X_\infty=\sum_{n=1}^{\infty} \beta_1^{n-1}\, w_n$. Clearly, $X_\infty$ is zero-mean normal with variance $(1-\beta_1)^{-2}$.

The LLR can be written as
\[
\lambda_{k+n}^k = \sum_{i=k+1}^{k+n} g (X_i, X_{i-1}),
\]
where 
\[
g(y,x) = \log \brcs{ \frac{(y-\beta_0 x)^2 -(y-\beta_1 x)^2}{2}}.
\]
Define
\[
\tilde{g} (x) = \int_{-\infty}^\infty g(y,x) \varphi(y-\rho_1 x) dy = \frac{(\beta_1-\beta_0)^2 x^2}{2} .
\]
We have
\begin{equation}\label{sec:Ex.2-1}
\sup_{y,x\in(-\infty,\infty)} \frac{\vert g(y,x)\vert}{1+|y|^2+|x|^2} \le Q
\quad\mbox{and}\quad \sup_{x\in(-\infty,\infty)} \frac{\tilde{g}(x)}{1+|x|^2} \le Q,
\end{equation}
where 
$$
Q=\max\set{1,\frac{\vert \beta_1^{2}-\beta_0^{2}\vert+(\beta_1-\beta_0)^{2}+1}{2}} .
$$ 
Now, define the Lyapunov function $V(x)= Q(1+|x|^2)$. Obviously, 
$$
\lim_{|x|\to\infty} \frac{\Eb_{x,0} V(X_1)}{V(x)} =\lim_{|x|\to\infty} \frac{1+\Eb |\beta_1 x+ w_{1}|^2}{1+|x|^2} =\beta_1^2<1,
$$
where $\Eb_{x,0}$ is the expectation under $\Pb_0(\cdot | X_0=x)$. Therefore, for any $\beta_1^2<\delta<1$ there exists $D>0$ such that condition $({\rm C}_1)$ in Section 5 in \cite{PergTar} 
holds with $C=[-n,n]$ for all $n\ge1$.  Next, by the ergodicity properties, for all $r\ge 1$,
\begin{equation}\label{sec:Ex.2-3n}
\lim_{n\to\infty}\Eb_{x,0} \vert X_n\vert^{r}=\Eb\vert X_\infty \vert^{r}<\infty \quad \text{for any}~ x\in(-\infty,\infty),
\end{equation}
where finiteness of $\Eb\vert X_\infty \vert^{r}$ for all $r\ge 1$ follows from the fact that $X_\infty$ is a Gaussian random variable.  Observe that under $\Pb_{x,0}$ for any $n\ge 1$
$$
X_n= \beta_1^n x +\sum_{i =1}^n \beta_1^{n-i} w_i,
$$
and hence, for any $r \ge 1$,
$$
\Eb_{x,0} \vert X_n \vert^{r}\le 2^{r} \left(\vert x\vert^{r} + \Eb_{0,0} \vert X_n \vert^{r}\right).
$$
Using \eqref{sec:Ex.2-3n}, we obtain that for some $C^{*}>0$
$$
M^{*}(x)=\sup_{n\ge 1} \Eb_{x,0} \vert X_n \vert^{r} \le C^{*} (1+\vert x\vert^{r}) \quad \text{and} \quad \sup_{n\ge 1}\Eb_0 M^{*}(X_n)<\infty.
$$
Therefore, the upper bounds in \eqref{sec:Ex.2-1} imply condition (${\rm C}_2$) in Section 5 in \cite{PergTar}. 

By  a slight extension of Theorem~5.1 in \cite{PergTar} to $r>1$, 
\[
\sum_{n=1}^\infty n^{r-1} \sup_{k\in\Zbb_+} \Pb_k\brc{\abs{\frac{1}{n} \lambda_{k+n}^k - I} > \varepsilon} <\infty \quad \text{for all}~ r\ge 1 ~ \text{and}~ \varepsilon >0,
\]
where 
$$
I=  \int_{-\infty}^\infty \brc{\int_{-\infty}^\infty g(y, x) \,\varphi(y-\beta_1 x)d y}\, \G(\D x),
$$
i.e., $n^{-1} \lambda_{k+n}^k$ converges $r$-completely to~$I$ as $n\to\infty$ under $\Pb_k$ for all $k\in\Zbb_+$ and all $r\ge 1$ (and even uniformly $r$-completely).
Since the stationary distribution $\G(x) =\Pb(X_\infty \le x)$  of the Markov process $X_n$ under $\Pb_0$ is normal  $\Nc(0,(1-\beta_1)^{-2})$, performing integration we obtain 
\[
I = \frac{(\beta_1 - \beta_0)^2}{2(1-\beta_1^2)} . 
\]
By Theorem~\ref{Th:FOAOgen} and Theorem~\ref{Th:FOAOgen2}, the Shiryaev detection procedure minimizes asymptotically $\Eb^\pi[(T-\nu)^m|T>\nu]$ for all $m\ge 1$, and the results
of Section~\ref{sec:SR} for the SR procedure can also be applied to all positive moments of the detection delay.
 \end{example}

\section{Concluding remarks and discussion}\label{sec:Conclusion}

1. The performed study shows that the famous Shiryaev and Shiryaev--Roberts change detection procedures have certain interesting asymptotic properties in the Bayesian context. Specifically, the 
Shiryaev procedure is asymptotically optimal (when the probability of false alarm is small) with respect to moments of the detection delay up to order $r\ge1$ for general non-iid models under 
mild conditions. These conditions are expressed via the SLLN for the LLR process and a rate of convergence ($r$-complete convergence). The $r$-complete convergence is usually not difficult 
to check in particular applications and examples. On the other hand, the $r$-quick convergence condition previously used in \cite{TV2005} is  stronger, and more importantly, 
usually much more difficult to verify.

2. A detailed examination of the proofs in Section~\ref{ssec:Gencase} shows that the Shiryaev procedure $T_{A_\alpha}$ minimizes not only the ``average'' moments 
$\Eb^\pi[(T-\nu)^r|T>\nu]$ but also conditional moments
$\Eb_\nu[(T-\nu)^r|T>\nu]$ uniformly for all (fixed) change points $\nu=0,1,2,\dots$ in class $\class$ asymptotically as $\alpha\to0$. Specifically, with an additional effort it can be established
that under the conditions of Theorem~\ref{Th:FOAOgen} for all $\nu\in\Zbb_+$ as $\alpha\to0$
\[
\inf_{T\in\class} \Eb_\nu[(T-\nu)^r| T>\nu] \sim \Eb_\nu[(T_{A_\alpha}-\nu)^r|T_{A_\alpha} >\nu] \sim \brc{\frac{|\log\alpha|}{I+\mu}}^r .
\]

3. The study of the Shiryaev--Roberts procedure shows that it is suboptimal in the Bayesian problem if the prior distribution has an exponential tail, 
but remains asymptotically optimal when the tail is heavy or if the parameter that characterizes the exponential tail goes to zero. 
This is expected since the SR procedure does not use the given prior distribution. 

4.  \cite{LaiIEEE98} proved asymptotic optimality of the CUSUM procedure with respect to the expected detection delay $\sum_{k=0}^\infty \pi_k^\alpha \, \Eb_k(T-k)^+$ in class $\class$ under the following
essential supremum condition  
\begin{equation*}
\lim_{n\to\infty}\,\sup_{\ell \ge k}\, \esup\, \Pb_k \left(\lambda_{\ell+n}^\ell -I \le -\varepsilon  |  \Fc_\ell \right)=0  \quad \text{for all}~\varepsilon>0.
\end{equation*}
However, on one hand, this condition is much more difficult to verify than the complete convergence condition 
required in our theorems. On the other hand, for many interesting models (including Markov and hidden Markov models) Lai's condition does not hold, while the complete convergence condition 
holds. This is the case, e.g., in Example~\ref{Ex3}.



\appendix

\renewcommand{\theequation}{A.\arabic{equation}}
\setcounter{equation}{0}

\section*{Appendix: Proofs of Lemmas}

\begin{proof}[Proof of Lemma~\ref{Lem:LB}]
For $\varepsilon\in(0,1)$, define $N_{\alpha,\varepsilon}= (1-\varepsilon) |\log\alpha|/(I+\delta)$. By the Chebyshev inequality,
\[
\begin{aligned}
\Eb^\pi[( T-\nu)^m |  T >\nu] &\ge \Eb^\pi[( T-\nu)^+]^m  \ge N_{\alpha,\varepsilon}^m \Pb^\pi( T  -\nu > N_{\alpha,\varepsilon})
\\
& \ge  N_{\alpha,\varepsilon}^m\brcs{\Pb^\pi( T>\nu)-\Pb^\pi(0<  T -\nu <N_{\alpha,\varepsilon})} 
\end{aligned}
\]
where
\[
\Pb^\pi(0<  T -\nu < N_{\alpha,\varepsilon}) = \sum_{k=0}^\infty \pi_k \Pb_k(k <  T < k + N_{\alpha,\varepsilon}).
\]
Since for any $ T\in\class$, $\Pb^\pi( T>\nu) =1-\PFA( T) \ge 1-\alpha$, we obtain
\begin{equation} \label{A1}
\inf_{ T\in\class}\Eb^\pi[( T-\nu)^m |  T >\nu]  \ge N_{\alpha,\varepsilon}^m\brcs{1- \alpha - \sup_{ T\in \class}\Pb^\pi(0 <  T-\nu < N_{\alpha,\varepsilon})}.
\end{equation}
Thus, to prove the lower bound \eqref{LBinclass} we need to show that 
\begin{equation}\label{Psupclasszero}
\lim_{\alpha\to0} \sup_{ T\in \class}\Pb^\pi(0 <  T-\nu < N_{\alpha,\varepsilon}) =0.
\end{equation}

To this end, introduce
\[
\begin{aligned} 
U_{\alpha,\varepsilon}^{k}( T) = e^{(1+\varepsilon) I N_{\alpha,\varepsilon}}\Pb_\infty\brc{k <  T < k+ N_{\alpha,\varepsilon}}, \quad
\beta_{\alpha,\varepsilon}^{k}  = \Pb_k\brc{\frac{1}{N_{\alpha,\varepsilon}} \max_{1\le n \le N_{\alpha,\varepsilon}} \lambda_{k+n}^k \ge (1+\varepsilon) \, I} .
\end{aligned}
\]
By inequality (3.6) in \cite{TV2005},
\begin{equation}\label{Pktauupper}
 \Pb_k\brc{k <  T < k+ N_{\alpha,\varepsilon}} \le  U_{\alpha,\varepsilon}^{k}( T)  + \beta_{\alpha,\varepsilon}^{k} .
\end{equation}
It is easy to see that 
\[
\sup_{ T\in\class} \Pb_\infty( T \le k) \le \alpha/\Pb(\nu \ge k), \quad k \ge 1 
\]
(cf. (3.8) in \cite{TV2005}), so that 
\begin{equation*}
\begin{aligned}
U_{\alpha,\varepsilon}^{k}( T) & \le e^{(1+\varepsilon) I N_{\alpha,\varepsilon}} \Pb_\infty( T <k+N_{\alpha,\varepsilon}) \le \alpha \, e^{(1+\varepsilon) I N_{\alpha,\varepsilon}}/\Pb(\nu > k+N_{\alpha,\varepsilon})
\\
& \le \exp\set{(1+\varepsilon) I N_{\alpha,\varepsilon}-|\log\alpha| -( k+N_{\alpha,\varepsilon}) \frac{\log \Pb(\nu > k+N_{\alpha,\varepsilon})}{ k+N_{\alpha,\varepsilon}}} .
\end{aligned}
\end{equation*}
By condition \eqref{Prior}, for all sufficiently large $N_{\alpha,\varepsilon}$ (small $\alpha$), there exists a (small) $\delta$ such that
\[
- \frac{\log \Pb(\nu > k+N_{\alpha,\varepsilon})}{ k+N_{\alpha,\varepsilon}} \le \mu + \delta.
\]
Hence, for a sufficiently small $\alpha$,
\begin{equation*}
\begin{aligned}
\sup_{ T\in\class} U_{\alpha,\varepsilon}^{k}( T) & \le  \exp\set{(1+\varepsilon) I N_{\alpha,\varepsilon}-|\log\alpha| +( k+N_{\alpha,\varepsilon}) (\mu+\delta)} 
\\
&\le \exp\set{-\varepsilon^2 |\log\alpha| + (\mu+\delta) (k+N_{\alpha,\varepsilon})},
\end{aligned}
\end{equation*}
which approaches zero as $\alpha\to0$ for $k\le K_{\alpha,\varepsilon}= \varepsilon_1 \varepsilon^2 |\log\alpha|/(\mu+\delta)$, 
where $0<\varepsilon_1<1$ and $\delta\to0$ as $\alpha\to0$. By condition \eqref{Pmaxgen}, 
$\beta_{\alpha,\varepsilon}^{k}\to0$ for all $k\in\Zbb_+$, and therefore, we obtain
\begin{align}\label{ProbUpper1}
 \sup_{ T\in \class}\Pb^\pi(0 <  T -\nu < N_{\alpha,\varepsilon}) & = \sum_{k=0}^\infty  \pi_k  \sup_{ T\in\class}  \Pb_k\brc{k <  T < k+ N_{\alpha,\varepsilon}} \nonumber
 \\
 &\le \Pb(\nu > K_{\alpha,\varepsilon}) +  \sum_{k=0}^{K_{\alpha,\varepsilon}}  \pi_k  \beta_{\alpha,\varepsilon}^{k} + 
 \max_{0 \le k \le K_{\alpha,\varepsilon}} \sup_{ T\in\class} U_{\alpha,\varepsilon}^{k}( T) \nonumber
 \\
 & \le \Pb(\nu > K_{\alpha,\varepsilon}) +  \sum_{k=0}^{K_{\alpha,\varepsilon}}  \pi_k  \beta_{\alpha,\varepsilon}^{k} + \exp\set{-\varepsilon^2 |\log\alpha| + (\mu+\delta) K_{\alpha,\varepsilon}},
 \end{align}
where all three terms go to zero as $\alpha\to0$ for all $\varepsilon>0$, so that \eqref{Psupclasszero} follows and the proof of the lower bound \eqref{LBinclass} is complete.   

The proof of the lower bound  \eqref{LBTA1} is essentially similar. Indeed, recall that, by \eqref{PFAineq}, $T_A\in\class$ if $A=1/\alpha$, so that it suffices to replace $\alpha$ by $1/A$ 
in the above argument. The details are omitted.
\end{proof}

\begin{proof}[Proof of Lemma~\ref{Lem:UpperEk}]
Obviously, for any $n >k$,
\[
\log \Lambda_n \ge \log\brc{\frac{\pi_k}{\Pb^\pi(\nu \ge n)} e^{\lambda_n^k}} = \lambda_n^k + \log \pi_k -\log \Pb(\nu \ge n),
\]
and hence,  for every $A >0$,
\[
(T_A-k)^+ \le  \tau_A^{(k)} : = \inf\set{n\ge 1: \lambda_{k+n}^k +|\log\Pb(\nu\ge k+n)| \ge \log(A/\pi_k)}, \quad k\in\Zbb_+,
\]
and $\Eb_k[(T_A-k)^+]^r\le \Eb_k (\tau_A^{(k)})^r$. 

Let $N_A=1+\lfloor \log (A/\pi_k) /(I+\mu-\varepsilon) \rfloor$, where $\lfloor x \rfloor$ is the greatest integer less or equal to $x$. For any $k\in\Zbb_+$, we have 
\begin{align}\label{Ektauineq}
\Eb_{k}\brcs{(T_A-k)^+}^r  & \le \Eb_k\brc{\tau_A^{(k)}}^r= \int_0^\infty r t^{r-1} \Pb_k\brc{\tau_A^{(k)} > t} \, \D t   \nonumber
\\
& \le N_A^r + \sum_{n=0}^{\infty} \int_{N_A+n}^{N_A+n+1} r t^{r-1}  \Pb_k \brc{\tau_A^{(k)} > t} \, \D t  \nonumber
\\
& \le  N_A^r + \sum_{n=0}^{\infty} \int_{N_A+n}^{N_A+n+1} r t^{r-1}  \Pb_k\brc{\tau_A^{(k)} > N_A+n} \, \D t  \nonumber
\\
& = N_A^r + \sum_{n=0}^{\infty} [(N_A+n+1)^r- (N_A+n)^r ] \Pb_k\brc{\tau_A^{(k)} > N_A+n} \nonumber
\\
& = N_A^r + \sum_{n=N_A}^{\infty} [(n+1)^r-n^r]  \Pb_k\brc{\tau_A^{(k)} > n}  \nonumber
\\
 & \le  N_A^{r} +\sum_{n=N_A}^{\infty}   r (n+1)^{r-1}   \Pb_k\brc{\tau_A^{(k)} >  n} \nonumber
 \\
 & \le N_A^{r} +\sum_{n=N_A}^{\infty}   r 2^{r-1} n^{r-1}   \Pb_k\brc{\tau_A^{(k)} >  n}.
 \end{align}
It is easily seen that  for all $k\in\Zbb_+$ and $n\ge N_A$
\begin{equation}\label{Probktau}
\begin{aligned}
\Pb_k\brc{ \tau_A^{(k)} >n} & \le \Pb_k\brc{\frac{1}{n} \lambda_{k+n}^k < \frac{1}{n} \log (A/\pi_k) - \frac{1}{n} |\log\Pb(\nu\ge k+n)|}
\\
& \le \Pb_k\brc{\frac{1}{n}\lambda_{k+n}^k <I + \mu -\varepsilon -\frac{1}{n} |\log\Pb(\nu\ge k+n)} .
\end{aligned}
\end{equation}
Since by condition ($ \Cb$), $N_A^{-1} |\log\Pb(\nu\ge k+N_A)| \to \mu$ as $A\to \infty$, for a sufficiently large value of $A$ there exists a small $\delta=\delta_A$ ($\delta_A\to 0$ as $A\to\infty$) 
such that $|\mu - |\log\Pb(\nu\ge k+N_A)|/N_A| < \delta$.  Hence, for all sufficiently large $A$, 
\begin{equation}\label{Probktau1}
\Pb_k\brc{\tau_A^{(k)} >n} \le \Pb_k\brc{\frac{1}{n}\lambda_{k+n}^k <I -\varepsilon - \delta} .
\end{equation}
Using \eqref{Ektauineq} and \eqref{Probktau1}, we obtain
\[
\begin{aligned}
 \Eb_k\brc{\tau_A^{(k)}}^r & \le N_A^r + r 2^{r-1} \, \sum_{n=N_A}^\infty n^{r-1} \Pb_k\brc{\frac{1}{n} \lambda_{k+n}^k < I  - \varepsilon -\delta}
 \\
 &\le  N_A^r + r 2^{r-1} \, \sum_{n=1}^\infty n^{r-1} \Pb_k\brc{\frac{1}{n} \lambda_{k+n}^k < I  - \varepsilon -\delta}.
 \end{aligned}
\]
This implies inequality \eqref{Ekineq} in Lemma~\ref{Lem:UpperEk}.

If $\pi_k=\rho(1-\rho)^k$ is geometric, condition \eqref{Prior} holds for all $n\ge 1$ with $\mu=|\log(1-\rho)|$, so that    
\[
\log(A/\pi_k)= \log (A/\rho) + \mu \, k \quad \text{and} \quad |\log\Pb(\nu\ge k+n)| = - \log(1-\rho)^{k+n}= (k+n) \mu.
\]
Therefore, the Markov times $\tau_A^{(k)}$, $k=0,1,\dots$ can be written as
\[
 \tau_A^{(k)} = \inf\set{n\ge 1: \lambda_{k+n}^k + \mu \, n  \ge \log(A/\rho)} ,
\] 
and inequality \eqref{Probktau} reduces to
\[
\Pb_k\brc{\tau_A^{(k)} >n} \le \Pb_k\brc{\frac{1}{n}\lambda_{k+n}^k <I  -\varepsilon},
\]
which holds for all $n \ge N_A=1+\lfloor \log(A/\rho)/(I+\mu-\varepsilon) \rfloor$ and all $0<\varepsilon < I +\mu$. Using this inequality and inequality \eqref{Ektauineq} with 
$N_A=1+\lfloor \log(A/\rho)/(I+\mu-\varepsilon) \rfloor$ yields inequality \eqref{EkineqGeom} in Lemma~\ref{Lem:UpperEk} and the proof is complete.
\end{proof}

\begin{proof}[Proof of Lemma~\ref{Lem:LB2}]
Let $N_{\alpha,\varepsilon}= (1-\varepsilon) |\log\alpha|/(I+\mu+\delta)$. The rest of the notation is the same as in the proof of Lemma~\ref{Lem:LB} above.
Similarly to \eqref{A1},
\begin{equation} \label{A1new}
\inf_{ T\in\class}\Eb^{\pi^\alpha}[( T-\nu)^m |  T >\nu]  \ge N_{\alpha,\varepsilon}^m\brcs{1- \alpha - \sup_{ T\in \class}\Pb^{\pi^\alpha}(0 <  T-\nu < N_{\alpha,\varepsilon})}.
\end{equation}
By \eqref{ProbUpper1},
\[
 \sup_{ T\in \class}\Pb^{\pi^\alpha}(0 <  T-\nu < N_{\alpha,\varepsilon})   \le \Pb(\nu > K_{\alpha,\varepsilon}) +  \sum_{k=0}^{K_{\alpha,\varepsilon}}  \pi_k^\alpha  \beta_{\alpha,\varepsilon}^{k} 
 + \exp\set{-\varepsilon^2 |\log\alpha| + (\mu_\alpha+\delta_\alpha) K_{\alpha,\varepsilon}},
 \]
where $K_{\alpha,\varepsilon}= \varepsilon_1 \varepsilon^2 |\log\alpha|/(\mu_\alpha+\delta_\alpha)$ with $0<\varepsilon_1<1$ (in particular, we may take $\varepsilon_1=\varepsilon$). 
Obviously, the last term vanishes as $\alpha\to0$ for all $\varepsilon\in(0,1)$. 
By condition \eqref{Pmaxgen}, the middle term also goes to zero. By condition \eqref{Prior} on the prior, as $\alpha\to0$,
\[
-\log \Pb(\nu > K_{\alpha,\varepsilon}) \sim \mu_\alpha K_{\alpha,\varepsilon} \sim \varepsilon_1\varepsilon^2 |\log\alpha| \to \infty.
\]
Therefore, the first term $\Pb(\nu > K_{\alpha,\varepsilon})$ also approaches zero as $\alpha\to0$ for all $\varepsilon \in(0,1)$.
It follows that  
\begin{equation*}
\sup_{ T\in \class}\Pb^{\pi^\alpha}(0 <  T-\nu < N_{\alpha,\varepsilon}) \to 0 \quad \text{as}~ \alpha\to0
\end{equation*}
and using \eqref{A1new}, we obtain that for all $0<\varepsilon<1$ and $m>0$ as $\alpha\to0$
\[
\inf_{ T\in\class}\Eb^{\pi^\alpha}[( T-\nu)^m |  T >\nu]  \ge (1-\varepsilon)^m\brc{\frac{|\log\alpha|}{I}}^m (1+o(1)).
\]
Since $\varepsilon$ can be arbitrarily small, the lower bound \eqref{LBinclass2} follows.

To prove the lower bound  \eqref{LBTA2} it suffices to replace $\alpha$ in the above argument by $1/A$ and recall that  $T_A\in\class$ if $A=1/\alpha$.
\end{proof}


\end{document}